\documentclass[a4paper]{amsart}

\usepackage[english]{babel}
\usepackage[ansinew]{inputenc}
\usepackage{amsmath, amsthm, amssymb, graphicx, ifthen}
\usepackage{color}

\newcommand{\Zb}{\mathbb Z}
\newcommand{\Dc}{\mathcal D}
\newcommand{\Ec}{\mathcal E}

\definecolor{orange}{rgb}{1,0.5,0}

\newcommand{\inv}{^{-1}}

\newcommand{\Nb}{\mathbb N}

\newcommand{\red}[1]{{\color{red}#1}}

\newcommand{\blue}[1]{{\color{blue}#1}}
\newcommand{\changed}{}
\newcommand{\changedm}{}

\newcommand{\Sppaths}{\mathcal {SP}}
\newcommand{\Konst}{6}
\DeclareMathOperator{\free}{F}
\DeclareMathOperator{\Sym}{Sym}

\newcommand{\mOmega}{\Omega}
\newcommand{\pairOmega}{\Omega}
\newcommand{\hOmega}{\widetilde{\Omega}}

\DeclareMathOperator{\rk}{rank}

\DeclareMathOperator{\cc}{b_0}
\DeclareMathOperator{\betti}{b_1}
\DeclareMathOperator{\bad}{bad}
\DeclareMathOperator{\dist}{d}

\newenvironment{customenum}[1]%
{	

	\begin{enumerate}
}%
{\end{enumerate}

}

\newenvironment{customenum*}[1]%
{\begin{enumerate}}%
{\end{enumerate}}

\newtheorem{theorem}{Theorem}
\newtheorem{proposition}[theorem]{Proposition}
\newtheorem{lemma}[theorem]{Lemma}
\newtheorem{corollary}[theorem]{Corollary}

\theoremstyle{remark}
\newtheorem{obs}[theorem]{Observation}
\newtheorem*{remark}{Remark}
\newtheorem*{definition}{Definition}
\newtheorem{unfolding}{Unfolding of type}

\newtheorem{claim}{Claim}
\newtheorem*{claim*}{Claim}
\newtheorem{case}{Case}

\title{On the rank of Coxeter groups}

\author{Mathieu Carette}
\address{Mathieu Carette, D\'epartement de Math\'ematiques CP 216, Universit\'e Libre de Bruxelles, Boulevard du triomphe, 1040 Brussels, Belgium}
\email{mcarette@ulb.ac.be}
\thanks{The first author is a F.R.S.-FNRS research fellow (Belgium).}
\author{Richard Weidmann}
\address{Richard Weidmann, Department of Mathematics, Heriot-Watt University, Riccarton, Edinburgh EH14 4AS, Scotland, UK}
\email{R.Weidmann@ma.hw.ac.uk}
\begin{document}
	
	\begin{abstract} We show that the standard generating set of a Coxeter group is of minimal {cardinality} provided that the non-diagonal entries of the Coxeter matrix are sufficiently large.
	\end{abstract}

	\maketitle

	\section{Introduction}

	Fix a set $S$ of cardinality $n$. A Coxeter group is a group given by a presentation of type 
	\[\langle S \mid (st)^{m_{st}}, s,t \in S\rangle\]
	where $m_{st}=m_{ts}\in \mathbb N_{\ge 2}\cup\{\infty\}$ for $s\neq t \in S$ and $m_{ss}=1$ for $s \in S$. It follows in particular that $s^2=1$ for each $s \in S$, i.e.\ that Coxeter groups are generated by elements of order $2$. Thus the group is determined by the symmetric matrix $M=(m_{st})$ and we denote the group given by the above presentation by $W(M)$. A subset of a group $T \subset W$ is a Coxeter generating set of type $M$ if there is a bijective map $S \mapsto T$ that extends to an isomorphism $W(M) \to W$. It should be noted that $M$ and $|S|$ are not algebraic invariants of $W$ as shown by the following example:
	\[ W \left(\begin{tabular}{ccc} 1 & 3 & 2 \\ 3 & 1 & 2 \\ 2 & 2 & 1  \end{tabular} \right) \cong D_{6} \times \Zb/2\Zb \cong D_{12} \cong W\left(\begin{tabular}{cc} 1 & 6 \\ 6 & 1  \end{tabular}\right) \]
	The Coxeter groups considered in this paper however are \emph{skew-angled}, i.e.\ each $m_{st} \geq 3$. For such a group, it follows from \cite{MuhWeid} that all skew-angled Coxeter generating sets have the same cardinality.
	
	The Coxeter generating set $S$ is often not a generating set of minimal cardinality, the symmetric group $\Sym(n+1)$ for example is a 2-generated group that is also a Coxeter group with a Coxeter generating set of cardinality $n$. In fact all irreducible finite Coxeter groups are $2$-generated \cite{Conder}.

	It is the purpose of this paper to show that standard generating sets are of minimal cardinality provided that the $m_{st}$ are sufficiently large. Let the rank of a group $G$ be the smallest cardinality of a generating set of $G$. We show the following:

	\begin{theorem}\label{thm:main}
		Let $S$ be a set of cardinality $n$ and $M=(m_{st})_{s,t \in S}$ a Coxeter matrix over $S$. Suppose that $m_{st}\geq \Konst. 2^n$ for all $s \neq t \in S$. Then $\rk(W(M))=n$.
	\end{theorem} 

Some remarks are in order:

\begin{enumerate}
\item While the bound given in the above theorem is probably not the best possible it could at best be improved by replacing $\Konst$ by a smaller constant. Indeed in \cite{Weid1} Coxeter groups where constructed where $m_{st}\ge 2^{n-2}$ for $s\neq t$ but the standard generating set is not minimal. These examples were obtained by iterating an observation of Kaufmann and Zieschang \cite{KaufZie}. \label{rem:non-example}
\item When $m_{st}$ is even for all $s\neq t$ then a simple homology argument shows that the standard generating set maps onto $\left(\Zb/2\Zb\right)^n$ and is therefore of  minimal cardinality.
\item If all $m_{st}$ are divisible by a common prime then the homological methods of Lustig and Moriah \cite{LustMor} again show that the Coxeter generating set is minimal, see \cite{KaufZie}.
\item The rank of Coxeter groups with Coxeter rank $3$ was computed by Klimenko and Sakuma in \cite{KlimSak}. In fact they classified all 2-generated discrete subgroups of $\hbox{Isom}(\mathbb H^2)$.
\item The rank of Coxeter groups that can be written as trees of dihedral groups was computed in \cite{Weid2}.
\item It follows from a result of Petersen and Thom \cite{PetThom} that $\hbox{rank }W(M)\ge\frac{n}{2}$ provided that $\frac{1}{2}\sum\limits_{s \neq t}\frac{1}{m_{st}}<1$.
\end{enumerate}

	This paper is organized as follows. We introduce a way to encode subgroups of a group by graphs, along with folding moves and Arzhantseva-Ol'shanskii moves in Section~\ref{sec:foldings}. An informal discussion of the main ideas of the proof of Theorem~\ref{thm:main} is given in Section~\ref{sec:outline}, followed by a discussion of the examples mentioned in remark~\eqref{rem:non-example} \changedm{above}. Motivated by the outline, Section~\ref{sec:reformulation} introduces tame marked decompositions and reformulates Theorem~\ref{thm:main} in terms of complexities of decompositions. Then tame marked decompositions of minimal complexity are investigated in Sections~\ref{sec:folding_sequence} and~\ref{sec:folded_graph}, leading to a proof of Theorem~\ref{thm:main}.
	
	\section{Graphs representing subgroups} \label{sec:foldings}
		A \emph{graph} $\Theta$ consists of a vertex set $V\Theta$, an edge set $E\Theta$ with an involution $^{-1}:E\Theta\to E\Theta$ without fixed point and two maps $\alpha,\omega:E\Theta\to V\Theta$ such that $\alpha(e)=\omega(e^{-1})$ for all $e\in E\Theta$. We call $\alpha(e)$ the \emph{initial vertex} of $e$ and $\omega(e)$ the \emph{terminal vertex} of $e$. For an (edge) path $\gamma=e_1,\ldots ,e_k$ we put $\alpha(\gamma):=\alpha(e_1)$ and $\omega(\gamma):=\omega(e_k)$. 

		If $S$ is a subset of some group $G$ that is closed under inversion then an \emph{$S$-labeling} of a graph $\theta$ is a map $\ell:E\Theta\to S$ such that $\ell(e^{-1})=\ell(e)^{-1}$ for all $e\in E\Theta$. Note that if all elements of $S$ are of order two, in particular if $G$ is a Coxeter group and $S$ a Coxeter generating set, then $\ell(e^{-1})=\ell(e)$ for all $e\in E\Theta$. This will be the case we are mostly interested in. For an edge path $\gamma=e_1,\ldots ,e_k$ we put $\ell(\gamma)=\ell(e_1)\cdot\ldots\cdot \ell(e_k)$.

		Given an $S$-labeled graph as above and a base vertex $v_0\in \Theta$ there is a homomorphism $$\mu:\pi_1(\Theta,v_0)\to G$$ given by $$[\gamma]\mapsto \ell(\gamma)$$ and we call $\mu(\pi_1(\Theta,v_0))$ the \emph{subgroup of $G$ represented by $(\Theta,v_0)$}. We will mostly be interested in the situation where $\mu$ is surjective and $\Theta$ is connected.

		Let now $G$ be a group and $S=S^{-1}$ be a generating of $G$. For any generating set $X=\{x_1,\ldots ,x_k\}$ of a subgroup $U$ of $G$ we can define an $S$-labeled graph $(\Theta_X,v_0)$ that represents $U$ as follows, see Figure~\ref{fig:wedge}.

		\begin{enumerate}
			\item For $1\le i\le k$ choose a word $w_i$ in $S$ of length $l_i$ that represents $x_i$ in $G$.
			\item Let $\Theta_X$ be the wedge of $k$ loops of length $l_1,\ldots ,l_k$ joined at the vertex $v_0$.
			\item The $i$-th \changedm{loop} is labeled by the word $w_i$ for $1\le i\le k$.
		\end{enumerate}
		\begin{figure}[htb]
			\center\input{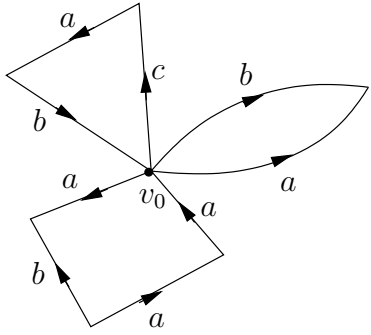}
			\caption{A labeled graph $(\Theta_X,v_0)$ representing $U=\langle X\rangle$ for $X= \{ab^{-1},cab,ab^{-1}a^2\}$}
			\label{fig:wedge}
		\end{figure}

		It is clear that the constructed $S$-labeled graph does indeed represent $U=\langle x_1,\ldots ,x_k\rangle$ as the obvious basis of $\pi_1(\Theta_X,v_0)$ gets mapped to $\{x_1,\ldots ,x_k\}$. Note that the orientation of an edge $e$ indicates whether the label drawn is the label of $e$ or of $e^{-1}$. As in subsequent sections all labeling are by elements of order $2$ we will be able to drop the orientation.

		We will now discuss two types of modifications that can be applied to an $S$-labeled graph $(\Theta,v_0)$ without affecting the represented subgroup. 
		
		\subsection{Stallings folds} \label{sec:stall_folds}
		The first modification is that of a Stallings fold \cite{Stall}. If there exist two edges $e_1,e_2\in E\Theta$ such that $\alpha(e_1)=\alpha(e_2)$ and $\ell(e_1)=\ell(e_2)$ then we define $\Theta'$ to be the labeled graph obtained from $\Theta$ by identifying $e_1$ and $e_2$ and labeling the new edge with $\ell(e_1)$ while leaving all other edges unchanged. Note that one or both of $e_1$ and $e_2$ may be loop edges. We denote the image of $v_0$ under the quotient map $v_0'$ and say that $(\Theta',v_0')$ has been obtained from $(\Theta,v_0)$ by a Stallings fold. 

		\begin{figure}[htb]
			\center\input{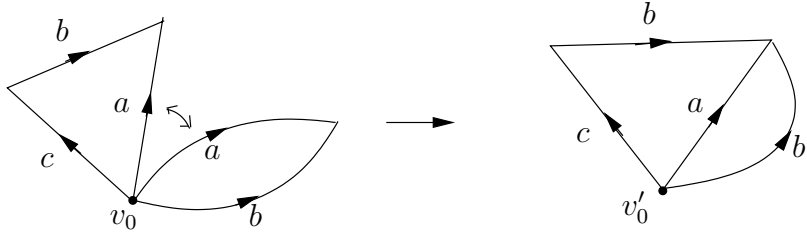}
			\caption{Two edges with label $a$ and initial vertex $v_0$ are identified}
			\label{fig:stallingsfold}
		\end{figure}
		
		We say that an $S$-labeled graph is \changedm{\emph{folded}} if no Stalling fold can be applied to it, i.e. if it has no two edges that have the same initial vertex and the same label. 

\smallskip If $\Theta'$ is obtained from $\Theta$ by a finite sequence of Stallings folds then there exists a natural surjection $\psi:\Theta\to\Theta'$. We record elementary but important observations, the same statements hold for labeled graphs with base vertices.
		\begin{lemma} \label{lemma:folded_graph} Let $\Theta$ be an $S$-labeled graph
			\begin{enumerate}
				\item If $\Theta_1, \Theta_2$ are folded graphs obtained from $\Theta$ by sequences of folds, then there is a (unique) label-preserving isomorphism $\varphi : \Theta_1 \to \Theta_2$ such that the surjections $\psi_i: \Theta \to \Theta_i$ satisfy $\psi_2 = \varphi \circ \psi_1$. \label{unique_folded}
				\item Suppose $\Theta$ is folded, then the label of any reduced path in $\gamma$ is reduced. If $S=X\dot\cup X^{-1}$ then the map $\mu : \pi_1(\Theta,v_0) \to \free(X)$ is injective.
			\end{enumerate}
		\end{lemma}
		Given a graph $\Gamma$, we denote by $\Theta^f$ the unique folded graph obtained from $\Theta$ by a sequence of folds. The uniqueness is a consequence of Lemma~\ref{lemma:folded_graph}~\eqref{unique_folded}. Note that $\Theta^f$ comes equipped with a canonical label-preserving map $\Theta \to \Theta^f$.
		
		\subsection{Arzhantseva-Ol'shanskii moves}\label{sec:AO-moves}
		The second type of modification was introduced by Arzhantseva and Ol'shanskii \cite{AO}, see also \cite{KapSchupp}, and we call it an \emph{AO-move}. Suppose that $(\Theta,v_0)$ is an $S$-labeled graph with base vertex and that  the following is given, see Figure~\ref{fig:AOmove}.

		\begin{enumerate}
			\item An edge path $\gamma=e_1,\ldots ,e_m$ in $\Theta$ (dotted line).
			\item A non-trivial subpath $\hat\gamma=e_i,\ldots ,e_j$ of $\gamma$ such that $\omega(e_p)$ is of valence $2$ and that $\omega(e_p)\neq v_0$ for $i\le p\le j-1$ (double dotted line)
			\item A word $w$ in $S$ of length $q$ such that $w=_G\ell(\gamma)$
		\end{enumerate}

		We then construct a new $S$-labeled graph $\Theta'$ by removing all edges and vertices of $\hat\gamma$ except the initial and terminal vertex of $\hat\gamma$. We then add a segment $\tilde \gamma$ of length $q$ by identifying the initial and terminal vertices of $\tilde \gamma$ with those of $\gamma$. We label $\tilde \gamma$ by the word $w$ and relabel $v_0$ by $v_0'$.

		\begin{figure}[htb]
			\center\input{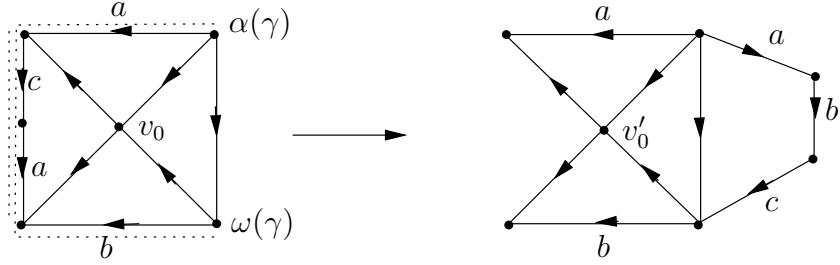}
			\caption{Here $w=abc=_Gacab^{-1}=\ell(\gamma)$}
			\label{fig:AOmove}
		\end{figure}

\medskip The following is a simple but important observation, see \cite{Stall} and \cite{AO} for details.

\begin{lemma} \label{lemma:fold_same_subgroup} Suppose that $(\Theta',v_0')$ is obtained from $(\Theta,v_0)$ by a Stallings fold or an AO-move. Then $(\Theta',v_0')$ and $(\Theta,v_0)$ represent the same subgroup of $G$.
\end{lemma}
				
		For a finite graph $\Gamma$ let $\cc(\Gamma)$ denote the number of connected components of $\Gamma$, let $\betti(\Gamma)$ be the first Betti number of $\Gamma$ and let $\chi(\Gamma)=\cc(\Gamma)-\betti(\Gamma)$ be the Euler characteristic of $\Gamma$.
		Let $\gamma$ be a path in a graph. An \emph{extremal edge of $\gamma$} refers to either the first or the last edge of $\gamma$, and all other edges of $\gamma$ are called \emph{inner edge of $\gamma$}. The \emph{inner} supbath of $\gamma$ is the subpath consisting of its inner edges. We further call $\alpha(\gamma)$ and $\omega(\gamma)$ the {\em extremal} vertices of $\gamma$.
		
	\section{Outline of the proof} \label{sec:outline}
		
		In this section we informally outline the general ideas of the proof of Theorem~\ref{thm:main}. As in the remainder of this paper, we fix a set $S$ of cardinality $n$.
		
		Recall that we aim to show that if $M$ is Coxeter matrix on $S$ whose entries are sufficiently large, then the associated Coxeter group $W = W(M)$ has rank $n$. The global strategy is the following:

\begin{enumerate}
\item Start with a generating set $X = \{x_1,\ldots,x_k\}$ of minimal cardinality $k = \rk(W)$. Represent $X$ by an $S$-labeled graph $\Theta_X$ as in Section~\ref{sec:foldings}. Note that 
the $S$-labeling induces a surjection $\pi_1(\Theta_X) \to W$. 

\item Then "simplify" $\Theta_X$ and produce a simpler graph $\Theta_1$ preserving the property that $\pi_1(\Theta_1) \to W$ is surjective. Keep on simplifying $\Theta_X = \Theta_0 \to \Theta_1 \to \Theta_2 \to \Theta_3 \to \ldots$ until { no further simplification is }possible. Observe that an appropriately defined complexity does not increase in this simplification process.
\item Derive the statement of the main theorem by evaluating and comparing the complexities of the initial and the terminal graph of the sequence.
\end{enumerate}

This strategy \changedm{can be used to prove} that $F_n$ is a group of rank $n$ and can also be used to prove Grushko's theorem. When proving that $F_n$ cannot be generated by fewer than $n$ elements this strategy is implemented in the following way: 

Let $F_n=F(a_1,\ldots ,a_n)$, $A=\{a_1,\ldots ,a_n\}$ and $S=A\cup A^{-1}$.  Let $X$ be a generating set of cardinality $k$. Construct $\Theta_X$ as above. The complexity is the first Betti number, in particular $b_1(\Theta_X)=k$ and the simplification process is to apply Stallings folds, in particular the terminal graph $\Theta^f$ is folded. Thus Lemma~\ref{lemma:folded_graph}~(2) and the solution to the word problem in the free group implies that the terminal graph is a rose with $n$ loop edges as any $a_i$ must be represented by a reduced closed path at the base vertex; thus $\Theta^f=\Theta_A$. The assertion now follows as Stallings folds do not increase the Betti number of a graph and therefore $$\rk F_n=k=b_1(\Theta_X)\ge b_1(\Theta^f)=b_1(\Theta_A)=n.$$

If the group is Coxeter group rather than a free group, then the final conclusion fails, i.e. one cannot argue that that the folded graph is isomorphic to $\Theta_S$; this is true as, other than in the free group, there are many reduced words in $S$ representing some given $s\in S$. Thus if $\Theta_i \neq \Theta_S$ is folded we need to find another way of simplifying~$\Theta_i$. 

Since $\Theta_i \neq \Theta_S$, there is an element $u \in S$ that is not the labeling of a loop edge at the basepoint. Since the map $\pi_1(\Theta_i) \to W$ is surjective, there is a closed path $p$ based at the basepoint such that $\ell(p) =_W \changedm{u}$. Choose $p_u$ to be a shortest such path. By the definition of $u$, there is no loop edge with label $u$ at the basepoint so $p_u$ is of length at least 2. Moreover, $p_u$ is reduced as it is chosen to be shortest, hence $\ell(p_u)$ is a reduced word as $\Theta_i$ is folded. Thus Lemma~\ref{lemma:almost_whole_rel} applies to $\ell(p_u)$ and implies the existence of a subpath $\gamma$ of $p_u$ whose label is almost a whole relator.
			
			The image of $\gamma$ in $\Theta_i$ can only take very few shapes since $\Theta_i$ is folded. Indeed, either $\gamma$ is embedded, or it takes one of the three shapes described in Figure~\ref{fig:shapes}. In this outline, we only treat the cases where $\gamma$ is embedded and where the image of $\gamma$ consist of a single loop edge and a long embedded path.
			\begin{figure}[htb]
				\center \input{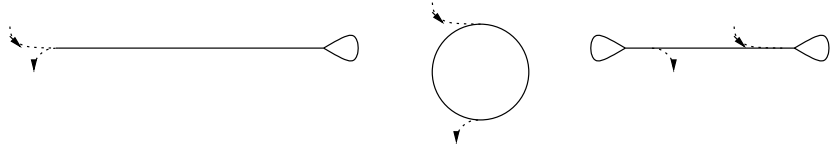}
				\caption{Possibilities for the image of $\gamma$}
				\label{fig:shapes}
			\end{figure}

Suppose that $\gamma$ is embedded in $\Theta_i$. As relators are sufficiently long compared to $n$ and since $\betti(\Theta_i) \leq n$ then $\gamma$ contains a subpath $\eta$ of length $4$ consisting of vertices of degree 2 in $\Theta_i$. Moreover the labeling of $\gamma$ is a relator minus three letters. Here we construct $\Theta_{i+1}$ by performing an AO-move removing $\eta$ and adding a path of length $3$ from $\alpha(\gamma)$ to $\omega(\gamma)$. This preserves the Betti number and simplifies the graph by decreasing the number of edges just as in the case of a Stallings fold.

\medskip Suppose now that the image of $\gamma$ consist of a single loop edge and a long embedded path, i.e. we are in the case illustrated in the first part of Figure~\ref{fig:shapes}. In order for this path to read almost a whole relation this embedded path must have length close to half a relation. The idea is to \changedm{``}translate" the loop edge along the long path, according to the relation $(stst\ldots stst)\inv t (stst\ldots stst) = s$ or $(stst\ldots stst)\inv t (stst\ldots stst) = t$ depending on whether $st$ is of odd or even order (see Figure~\ref{fig:translate}).
			\begin{figure}
				\center \input{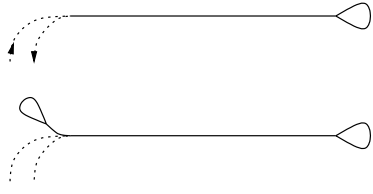}
				\caption{Translating a loop edge}
				\label{fig:translate}
			\end{figure}

		Thus  we handle this situation by keeping the initial loop edge and adding the new loop edge to $\Theta_{i+1}$. Note now that $\betti(\Theta_{i+1}) = \betti(\Theta_i) + 1$, i.e. our original complexity increases and the setup must be modified to take into account the fact that the two loops are translates of each other along some path. To do this we keep track of the translating path, later called a \emph{special path}. As a complexity we then replace the Betti number by $ \betti(\Theta_i) - |\Sppaths_i|$ where $\Sppaths_i$ is the collection of special paths. An additional problem \changedm{is} that this modification does not decrease the number of edges, i.e. does not simplify in the same way as before. However we can think of the edges that lie on special path as being "used up" to translate. In particular it will turn out the number of edges that are not used up decreases.

\smallskip It turns out that the actual complexity used in the proof is more subtle. Moreover the complexity as we define it will \changedm{incorporate} information that allows to conclude as outlined above but that also decreases by any of the modification we perform.  One of the additional problems that needs to be dealt with, and that is in fact responsible for the exponential bound, can be illustrated by discussing some examples from \cite{Weid1}. These examples are Coxeter groups such that the standard generating set $S = \{s_1, \ldots, s_n\}$ is not of minimal cardinality, but where each element of the Coxeter matrix is large, i.e.\ $m_{ij} \geq 2^{n-2}$ for $1\le i \neq j$ such that $1 \leq i,j\leq n$. We discuss an example in the case $n=5$, it will be clear that the example can be generalized to arbitrary $n$. 
			
			Choose any odd integer, say $101$. Consider the Coxeter matrix $M = (m_{ij})_{1 \leq i,j \leq 5}$ where $m_{ii} = 1$ for $1 \leq i \leq 5$, $m_{12} = 8$, $m_{2j} = 101$ for $j \geq 3$ and $m_{ij} = +\infty$ for all other entries. Consider the following subset of $W(M)$:
			\[X = \{s_2, \underbrace{s_1s_2s_1s_2s_1s_2s_1}_{\text{7 letters}}(s_3s_2)^{50}, \underbrace{s_1s_2s_1}_{\text{3 letters}}(s_4s_2)^{50}, \underbrace{s_1}_{\text{1 letter}}(s_5s_2)^{50} \}\]
			We claim that $X$ is a generating set for $W(M)$. Indeed, $S$ can be obtained by making products of elements of $X$ as shown by the sequence of $S$-labeled graphs in Figure~\ref{fig:non-example}. Therefore $\rk(W(M)) \leq 4$.
			
			\begin{figure}
				\center \input{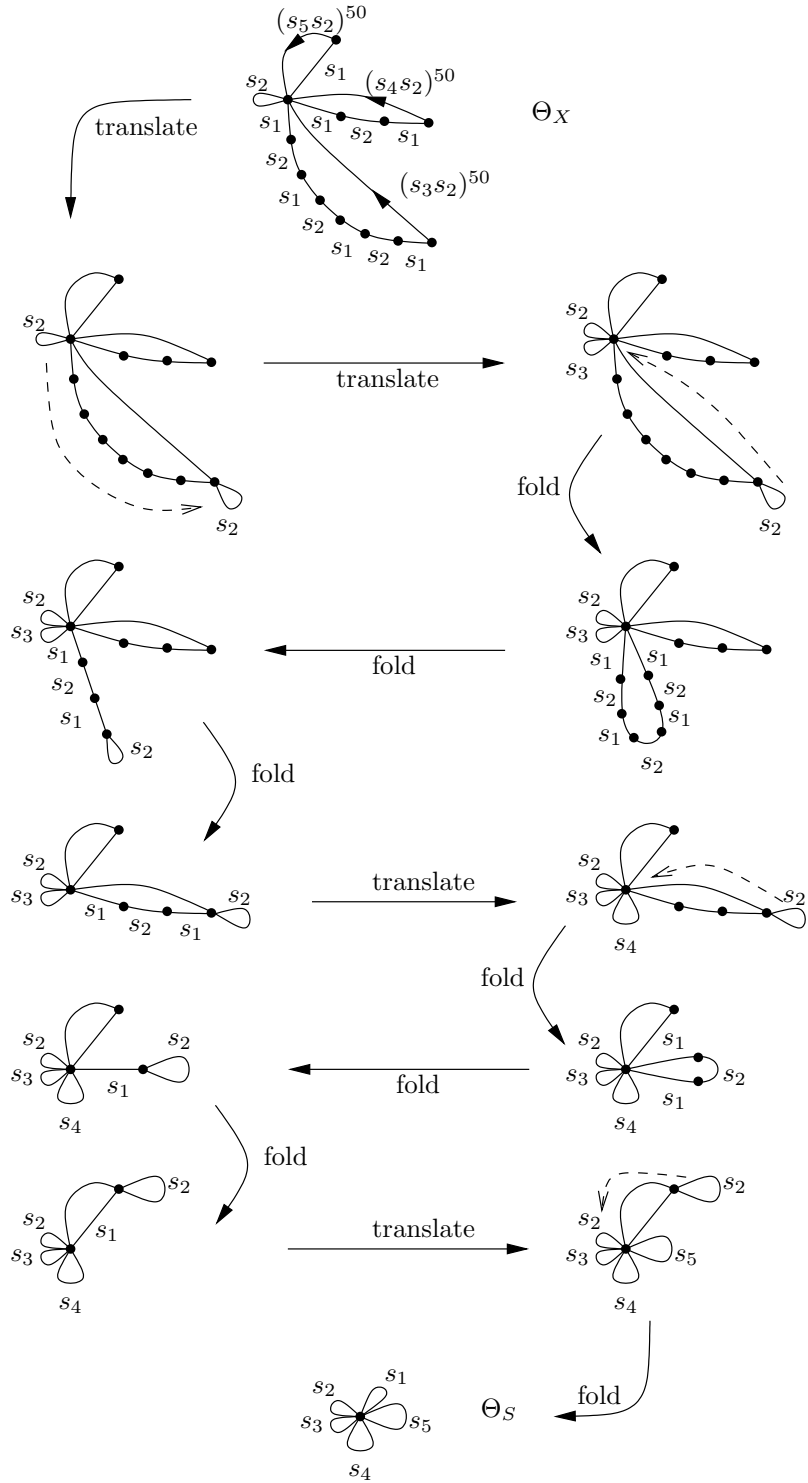}
				\caption{A Coxeter group with large labels but rank smaller than $|S|$.}
				\label{fig:non-example}
			\end{figure}
			
			\begin{remark} An explanation for this phenomenon is that the special path with labelling $s_1s_2s_1s_2s_1s_2s_1$ keeps getting \emph{halved}, eventually becoming a loop edge. In order to prevent this, we require that each special path remains long enough even after halving $n$ times (hence the exponential bound). The halving phenomenon is taken care of in the proof of Lemma~\ref{lemma:Delta_folded}, see Figure~\ref{fig:fold_Delta_half}.
			\end{remark}

	\section{Setup} \label{sec:reformulation}
	
	It is the purpose of this section to introduce the right setup and the appropriate notion of {complexity} needed for the proof of the main theorem. We further observe some simple consequences.

		\subsection{Definitions and conditions}\label{subsec:marked_decomp}
			Fix a set $S$ of cardinality $n$. Let $M$ be a Coxeter matrix over $S$. The $M$-special graphs keep track of loop edges, their translates and the special path along which they are translated.
			
			An \emph{$M$-special graph} (or more simply a special graph) $\Delta$ is an $S$-labeled graph such that the following holds (see Figure~\ref{fig:special_graph}). There is a collection of paths $\Sppaths$, called \emph{special paths} such that
			\begin{customenum}{$\Delta$}
				\item Each special path is of length at least $5$. \label{cond:Delta_length_sppath}
				\item Each non-loop edge lies in a special path, and there is no isolated vertex.
				\item Each special path is simple or simple closed.
				\item For two special paths $\delta \neq \delta'$ any component of $\delta\cap\delta'$ consists of a common extremal vertex and possibly (at most two) common extremal edges. \label{cond:sppaths_emb_or_closed}
				\item If the basepoint $v$ of a loop edge lies on a special path $\delta$, then $v$ is an extremal vertex of $\delta$. 
				\item \label{cond:Delta_relator} For each special path $\delta$ there are $s\neq t \in S$ and loop edges $e$ and $f$ based at $\alpha(\delta)$ and $\omega(\delta)$ such that the path $\delta,f,\delta\inv,e$ reads the relation $(st)^{m_{st}}$.
			\end{customenum}
			Note that the set $\{s,t\}$ in condition \eqref{cond:Delta_relator} is unique. Call this set the \emph{type} of the special path $\delta$. Let $\Ec$ be the set of loop edges of $\Delta$.
			\begin{remark} Condition \eqref{cond:Delta_relator} implies that $\ell(e) = \ell(f)$ if and only if $m_{st}$ is even.
			\end{remark}
			\begin{figure}
				\center \input{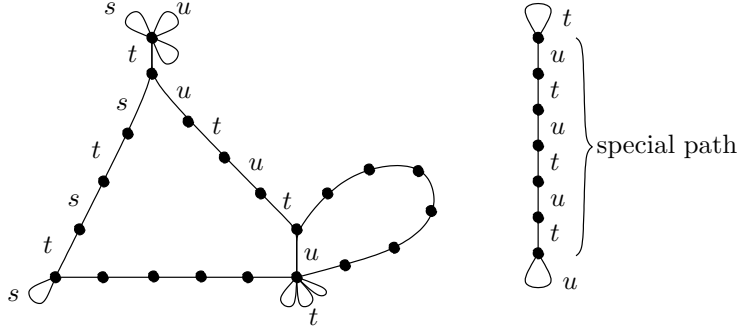}
				\caption{An $M$-special graph where $m_{st} = 6$ and $m_{tu} = 7$.}
				\label{fig:special_graph}
			\end{figure}
			
			Crucial for the setup are decompositions of $\Theta$.
			
			A \emph{decomposition} (of a labeled graph $\Theta$) is a tuple $\Dc = (M,\Gamma,\Delta,F,p,\Theta)$ where
			\begin{itemize}
				\item $M$ is a Coxeter matrix over $S$, with associated Coxeter group $W(M)$ generated by $S$. 
				\item $\Gamma$ is a (not necessarily connected) $S$-labeled graph.
				\item $\Delta$ is a (not necessarily connected) $M$-special graph.
				\item $F$ is a subgraph of $\Delta \backslash \Ec$ and and $p:F \to \Gamma$ is a label-preserving graph morphism. 
				\item $\Theta$ is the $S$-labeled graph obtained from the disjoint union $\Gamma \sqcup \Delta$ by identifying each vertex and edge of $F$ with its image under $p$.
			\end{itemize}
			
			\begin{figure}[htb]
				\center \input{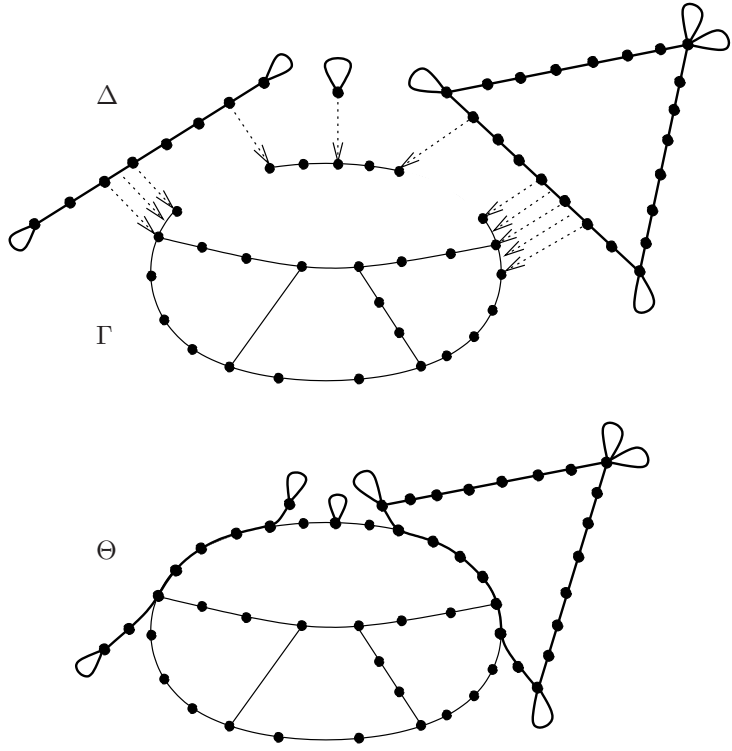}
				\caption{A decomposition of a graph $\Theta$}
				\label{fig:decomposition}
			\end{figure}

			When drawing a decomposition, we will draw $\Delta$ slightly thicker, and $p$ will be represented by dotted arrows from vertices and edges of $\Delta$ to $\Gamma$. Collapsing these dotted edges yields the graph $\Theta$, where images of edges of $\Delta$ are also drawn thicker. See Figure~\ref{fig:decomposition}.

\smallskip If more than one decomposition is considered, the components of a decomposition will consistently be denoted with the same decorations as the decomposition, for example $\Dc' = (M',\Gamma',\Delta',F',p',\Theta')$.

\smallskip			Note that any edge of $\Gamma\sqcup\Delta$ has a well-defined image in $\Theta = (\Gamma \sqcup \Delta) / \{x \sim p(x)\}$. Throughout this paper, we use the following convention: if $X$ is a subgraph of $\Gamma \sqcup \Delta$ then $\bar X$ denotes its image in $\Theta$.

			\smallskip A priori the way special paths intersect is very complicated 
			and therefore $F$ can be complicated as well.  We want to ensure that the graph $F \subset \Delta$ is a forest.  Thus we introduce tame markings to keep some control over $F$ and hence also over the embedding of special paths in $\Theta$. 
						

		\smallskip 	A \emph{marked decomposition}  is a pair $(\Dc,\pairOmega)$ where $\Dc$ is a decomposition and $\pairOmega$ is a subgraph of $\overline{\Delta \backslash \Ec}$. $\mOmega$ is called a \emph{marking}. We let $\hOmega$ denote the preimage of $\mOmega$ in $\Delta \backslash \Ec$.
			Denote by $\mOmega_k$ the $k$-neighborhood of $\mOmega$ in $\overline{\Delta \backslash \Ec}$ and let $\hOmega_k$ be the $k$-neighborhood of $\hOmega$ in $\Delta \backslash \Ec$. 
			\begin{remark} Clearly the image of $\hOmega_k \subset \Delta$ lies in $\mOmega_k \subset \bar \Delta$, i.e.\ $\overline{\hOmega_k}\subset \mOmega_k$. The opposite inclusion does not hold in general. However we will see that it does for $k=3$ provided that  condition \eqref{cond:Omega} defined below holds, see Lemma~\ref{homotopyOmega3}.
			\end{remark}
			
			\begin{definition} The \emph{potential} of a decomposition is
			\[c_* = \betti(\Theta) + \cc(\Delta) - |\Ec|.\]
			\end{definition}
			
			\begin{definition} A marked decomposition $(\Dc,\pairOmega)$ is \emph{tame} if the following conditions are satisfied.
			\begin{customenum*}{$\mathbf{\Theta}$}
				\item $\Theta$ is connected and the homomorphism $\pi_1(\Theta) \to W(M)$ induced by the $S$-labeling of $\Theta$  is surjective. \label{cond:pi1surj}
			\end{customenum*}
			\begin{customenum}{$\mathbf{\mOmega}$}
				\item If $x \neq y$ are vertices of $F$ such that $p(x) = p(y)$ then $x$ and $y$ lie in $\hOmega_3$. \label{cond:Omega_vertex}
				\item Each edge $e$ of $F$ is in $\hOmega_3$. \label{cond:Omega_edge}
				\item $|E\mOmega|\leq 8(\chi(\Delta) - \chi(\bar\Delta) - \chi(\mOmega))$ \label{cond:Omega_complexity}
				\item If $\gamma$ is a reduced path of length at most $8$ in $\overline{\Delta \backslash \Ec}$ having both endpoints in $\mOmega$, then $\delta$ is contained in $\mOmega$. \label{cond:Omega_path} 
			\end{customenum}
			\begin{customenum*}{$\mathbf{\bar \Delta \ast}$}
				\item Each special path $\delta$ has as many edges as its image $\bar\delta$ in $\Theta$. \label{cond:inj_edges_sppaths}
			\end{customenum*}
			\begin{customenum*}{$\mathbf{M}$} \item $m_{st}\geq \Konst.2^{c_*}$ for each $s \neq t \in S$. \label{cond:M_large}
			\end{customenum*}
			\end{definition}
			
			In order to simplify notation we also introduce the condition
			\begin{customenum*}{$\mathbf{\mOmega}$}
				\item \changed{Conditions} \eqref{cond:Omega_vertex},\eqref{cond:Omega_edge},\eqref{cond:Omega_complexity} and \eqref{cond:Omega_path} hold. \label{cond:Omega}
			\end{customenum*}
			
			Some remarks are in order. 
			\begin{itemize}
				\item It is obvious from the way $\Theta$ is obtained from $\Gamma, \Delta, F$ and $p$ that $\chi(\Theta) = \chi(\Gamma) + \chi(\Delta) - \chi(F)$.
				\item Note that if $\Theta$ is connected, condition \eqref{cond:pi1surj} is independent of the basepoint, hence we do not need to specify it.
				\item While all edges of $F$ lie in $\hOmega_3$, there might be isolated vertices of $F$ outside~$\hOmega_3$.
			\end{itemize}
			
The following lemma clarifies the relationship between $\hOmega_3$ and $\mOmega_3$. Note that \changed{later on} the condition that $\Delta$ is folded will always be satisfied\changed{, see Lemma~\ref{lemma:Delta_folded}}.
			
\begin{lemma}\label{homotopyOmega3} Let $(\Dc,\pairOmega)$ be a marked decomposition satisfying condition \eqref{cond:Omega}. Then the following hold:

\begin{enumerate}
\item If $k\in\{3,4\}$ then \changed{for} each component $\mOmega_k^i$ of $\mOmega_k$ the inclusion of $\mOmega_k^i\cap \mOmega$ into $\mOmega_k^i$ is a homotopy equivalence.
\item If $\Delta$ is folded then for each component $\hOmega_3^i$ of $\hOmega_3$ the inclusion of $\hOmega_3^i\cap \hOmega$ into $\hOmega_3^i$ is a homotopy equivalence.
\item If $\Delta$ is folded, $x\neq y\in F$ are vertices such that $p(x)=p(y)$ and $\gamma_x,\gamma_y\subset \Delta$ are shortest paths from $\hOmega$ to $x$, respectively $y$, then $\bar\gamma_x=\bar\gamma_y$. 
\item If $\Delta$ is folded then $\hOmega_3$ is the preimage of $\mOmega_3$ in $\Delta\backslash \Ec$.
\end{enumerate}
\end{lemma}

Note that saying that the inclusion of some connected graph $A$ into some  graph $B$ is a homotopy equivalence is the same as saying that $B$ is obtained from $A$ by attaching trees  along single vertices.

\begin{proof}
(1) The conclusion for $k=4$ clearly implies the conclusion for $k=3$ thus we restrict ourselves to the case $k=4$. Suppose that the conclusion does not hold for some component $\mOmega_4^i$. Then there exists a reduced path $\bar\gamma$ in $\mOmega_4^i$ such that $\alpha(\bar\gamma),\omega(\bar\gamma)\in\mOmega_4^ i\cap\mOmega$ and that all other vertices and all edges of $\bar\gamma$ lie in $\mOmega_4^i\backslash \mOmega$. As $\bar\gamma$ lies in $\mOmega_4$ we can moreover choose $\bar\gamma$ to be of length at most $8$. Condition \eqref{cond:Omega_path} then implies that $\bar\gamma$  lies in $\mOmega$, a contradiction.

(2) Suppose that the conclusion does not hold for some component $\hOmega_3^i$. Then there exists a reduced path $\gamma$ in $\hOmega_3^i$ such that $\alpha(\gamma),\omega(\gamma)\in\hOmega$ and that all other vertices and all edges of $\gamma$ lie in $\hOmega_3^i\backslash \hOmega$. As $\gamma$ lies in $\hOmega_3$ we can moreover choose $\gamma$ to be of length at most $6$. Let $\gamma'$ be the path obtained from $\bar\gamma$ by reduction, Condition \eqref{cond:Omega_path} then implies that $\gamma'$  lies in $\mOmega$ which the implies that $\gamma'$ is reduced to a point. This however implies that the label of $\gamma$ was not reduced contradicting the assumption that $\Delta$ is folded.

To see (3) note first that $x$ and $y$ lie in $\hOmega_3$ by condition \eqref{cond:Omega_vertex}. Thus $\gamma_x$ and $\gamma_y$ exist and of length at most $3$. Let now $\bar\gamma$ be the path obtained from $\bar\gamma_x\cdot\bar\gamma_y^{-1}$ by reduction. $\bar\gamma$ is a reduced path with endpoints in $\mOmega$ of length at most six and is therefore contained in $\mOmega$ by condition \eqref{cond:Omega_path}. As no edge of $\gamma_x$ and $\gamma_y$ lies $\hOmega$ it also follows that no edge of $\bar\gamma$ lies in $\mOmega$. Thus $\bar\gamma$ is reduced to a single point, i.e.\ $\bar\gamma_x=\bar\gamma_y$.

The proof of (4) is by contradiction. Thus we assume that there exists some vertex $v\in V\Delta$ such that $\dist_{\bar \Delta}(\bar v,\mOmega)\le 3$ and $\dist_{\bar \Delta}(\bar v,\mOmega)<\dist_{\Delta}(v,\hOmega)$. Among all such $v$ we choose $v$ such that $\dist_{\bar\Delta}(\bar v,\mOmega)$ is minimal. Let $\gamma=e_1,\ldots ,e_k$ be the shortest path from $\mOmega$ to $\bar v$, clearly $1\le k\le 3$. If $v\in\hOmega_3$ then it follows from (1) and (2) that $\dist_{\bar \Delta}(\bar v,\mOmega)=\dist_{\Delta}(v,\hOmega)$ as the path $\gamma_v$ from $\hOmega$ to $v$ maps injectively to a path from $\mOmega$ to $\bar v$. Thus $v\notin\hOmega_3$. It then follows from condition \eqref{cond:Omega_vertex} that $v$ is the unique lift of $\bar v$, in particular there exists a unique edge $f_k\in\Delta\backslash\Ec$ such that $\bar f_k=e_k$. It follows from our minimality assumption on $v$ that $\dist_{\Delta}(\alpha(f_k),\hOmega)=\dist_{\bar\Delta}(\alpha(e_k),\mOmega)=k-1$. This however implies that $$\dist_{\Delta}(v,\hOmega)=\dist_{\Delta}(v,\alpha(f_k))+\dist_{\Delta}(\alpha(f_k),\hOmega)=(k-1)+1=k=\dist_{\bar\Delta}(\bar v,\mOmega)$$ which contradicts the above assumption.
\end{proof}

			\begin{lemma} \label{lemma:O3} Let $(\Dc,\pairOmega)$ be a tame marked decomposition. Then
				\begin{enumerate}
					\item $\hOmega_3$ does not cover all inner edges of a special path. \label{lemma:O3_small}
					\item $\hOmega_3$ and $F$ are forests. \label{lemma:O3_forest}
				\end{enumerate}
			\end{lemma}
			\begin{proof} Statement \eqref{lemma:O3_forest} follows easily from statement \eqref{lemma:O3_small}. Indeed any closed path in $\Delta \backslash \Ec$ must contain all inner edges of a special path, so $\hOmega_3$ is a forest. By condition \eqref{cond:Omega_edge} all edges of $F$ are in $\hOmega_3$ so $F$ is a forest as well.
			
				\changed{Suppose that \eqref{lemma:O3_small} does not hold. Then there is a special path $\delta$ whose inner edges are all in $\hOmega_3$. 
By Lemma~\ref{homotopyOmega3} (1) $\mOmega$ covers all but possibly the first 3 and last 3 inner edges of $\bar \delta$. Thus condition \eqref{cond:inj_edges_sppaths} implies that $l(\delta)\le|E\mOmega| +8$.} 
				
				Since $F$ does not contain loop edges by definition, the set of loop edges $\Ec \subset \Delta$ is mapped injectively to $\Theta$, so that 
				\[ \chi(\Delta \backslash \Ec) - \chi(\overline{\Delta\backslash \Ec}) = \chi(\Delta) - |\Ec| - \chi(\bar \Delta) + |\Ec| = \chi(\Delta) - \chi(\bar\Delta).\]
				Combining this last observation with condition \eqref{cond:Omega_complexity} we have 
				\begin{align*}
					|E\mOmega| & \leq 8\left(\chi(\Delta) - \chi(\bar\Delta) - \chi(\mOmega) \right) = 8\left(\chi(\Delta \backslash \Ec) - \chi(\overline{\Delta\backslash \Ec}) - \chi(\mOmega)\right) \\
												& = 8\left(\left(\cc(\Delta \backslash \Ec) - \betti(\Delta \backslash \Ec)\right) + \left(\betti(\overline{\Delta\backslash \Ec}) - \cc(\overline{\Delta\backslash \Ec})\right)  + \left(\betti(\mOmega) - \cc(\mOmega)\right)\right) \\
												& \leq 8\left(\cc(\Delta \backslash \Ec) - 0 + \betti(\overline{\Delta\backslash \Ec}) - 1  + \betti(\mOmega) - 1\right)
				\end{align*}
				In the last inequality we assumed that $\cc(\mOmega)>0$ (and in particular ${\cc(\overline{\Delta \backslash \Ec})>0}$) as otherwise $\mOmega$ is empty and Lemma \ref{lemma:O3} is trivial.
				Note that $\cc(\Delta \backslash \Ec) = \cc(\Delta)$ and that both $\betti(\overline{\Delta\backslash \Ec})$ and $\betti(\mOmega)$ are smaller than $\betti(\Theta) - |\Ec|$ so
				\[ |E\mOmega| \leq 8\cc(\Delta) + 16(\betti(\Theta) - |\Ec| - 1) \leq 16(\underbrace{\cc(\Delta) + \betti(\Theta) - |\Ec|}_{=c_*} - 1) - 8\]
				and
				\[ \changed{l(\delta)  \leq |E\mOmega| + 8} \leq 16(c_*-1) - 8 + 8.\]
				On the other hand $l(\delta)=m_{st} - 1$ for some $s \neq t$, and by condition \eqref{cond:M_large} we have
				\[ l(\delta) = m_{st} - 1 \geq \Konst.2^{c_*} -1 \]
				In conclusion we have $16(c_*-1) \geq \Konst.2^{c_*} -1$ which is impossible for any integer $c_* \geq 1$. 
			\end{proof}

			
						
		\subsection{Complexity}
			Fix a decomposition $\Dc$. Recall from Section \ref{sec:stall_folds} that $\Theta^f$ denotes the unique folded graph obtained from $\Theta$ by a sequence of folds. Let $\Delta^f$ be the image of $\bar \Delta$ in $\Theta^f$.
			
			The \emph{complexity} of a decomposition $\Dc$ is the tuple $(c_1,c_2,\ldots, c_7) \in \Nb^7$ were
			\begin{itemize}
				\item (primary complexity)
				\begin{itemize}
					\item $c_1 = \betti(\Theta) - |\Sppaths|$
					\item $c_2 = \betti(\Theta) + \chi(\Delta)$
					\item $c_3 = |E(\Theta^f \backslash \Delta^f)|$
					\item $c_4 = |E\Theta^f|$ 
					\item $c_5 = |E\Delta|$
				\end{itemize}
				\item (secondary complexity)
				\begin{itemize}
					\item $c_6 = |E\Theta|$
					\item $c_7 = |E(\hOmega_3 \backslash F)|$
				\end{itemize}
			\end{itemize}
			Complexities are ordered according to the lexicographic order. The first five components form the \emph{primary complexity} while the last two components are called the \emph{secondary complexity}.
			
			\begin{remark} $c_1$ and $c_2$ may be negative for general decompositions, but they are non-negative provided $F$ is a forest, thus by Lemma~\ref{lemma:exist_minimalcpx_decomp} the complexities $c_1$ and $c_2$ are non-negative provided that $(\Dc,\pairOmega)$ is a tame marked decomposition.
			\end{remark}
			
			The following theorem is the main technical result of this article. It immediately implies Theorem \ref{thm:main}.
			\begin{theorem}\label{thm:reformulation}
				Let $(\Dc,\pairOmega)$ be a tame marked decomposition. Then $c_1 \geq n$.
			\end{theorem}
			Note that if $X$ is a generating set with $|X| = \rk(W(M))$, then there is a tame marked decomposition $(\Dc_X,\emptyset)$ with $c_1=\rk(W(M))$. Indeed, we choose $\Delta_X$ to be empty and $\Gamma_X = \Theta_X$ the wedge of circles described in Section \ref{sec:foldings}, see Figure \ref{fig:wedge}. Theorem \ref{thm:reformulation} applies and $\rk(W(M)) \geq n$. Thus Theorem \ref{thm:reformulation} does indeed imply Theorem \ref{thm:main}. \changedm{Theorem \ref{thm:reformulation} in turn is a consequence of Proposition \ref{proposition:Gammaf_is_R} as explained at the end of Section~\ref{sec:folded_graph}.}
			
		\subsection{Unfolding}
			For the remainder of this section, $\Dc=(M,\Gamma,\Delta,F,p,\Theta)$ is a decomposition of some connected labeled graph $\Theta$, $\Gamma$ is nonempty and $F$ is a forest. We discuss three ways to unfold $\Theta$ by changing $F$ and $\Gamma$ without changing the primary complexity. \changed{In each case we obtain a new decomposition $\Dc'=(M',\Gamma',\Delta',F',p',\Theta')$ where $M'=M$ and $\Delta'=\Delta$. We will not describe $\Theta'$ as it is determined by $\Gamma'$, $F'$ and $p'$.}
			
			
			\begin{unfolding}[See Figure \ref{fig:cc_vertex}] \label{unfold:cc_vertex} Choose a subset $F'$ of $F$ consisting of a single vertex in each connected component of $F$. \changed{Let $\Gamma' = \Gamma$ and let $p'$ be the restriction of $p$ to $F'$}. 
			\end{unfolding}
			\begin{figure}[htbp]
				\center \input{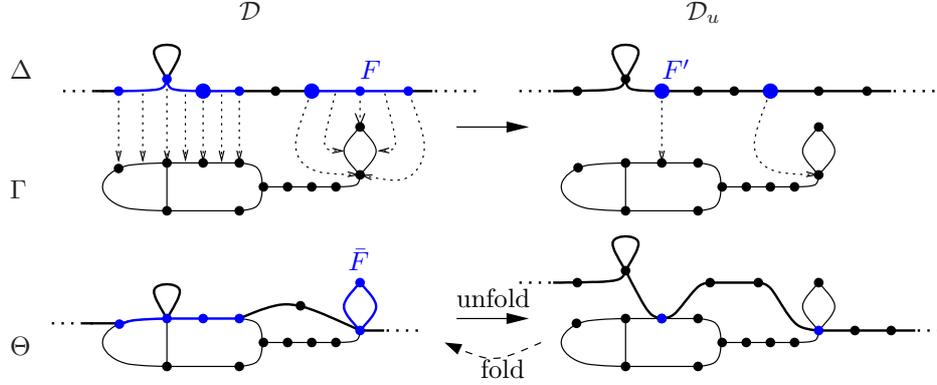}
				\caption{Unfolding of type \ref{unfold:cc_vertex}.}
				\label{fig:cc_vertex}
			\end{figure}
			\begin{unfolding}[See Figure \ref{fig:remove_cc}] \label{unfold:remove_cc} Suppose $T$ and $T'$ are distinct connected components of $F$ in the same connected component of $\Delta$. Let $x \in T$ and $y \in T'$. 
Choose a path $\gamma$ in $\Delta$ from $x$ to $y$. Let $\gamma'$ be a new path having the same length and labeling as $\gamma$. Glue $\gamma'$ to $\Gamma$ by identifying $\alpha(\gamma')$ with $p(x)$ and $\omega(\gamma')$ with $p(y)$. Let $\Gamma' = \Gamma \cup \gamma$, \changed{$F'=F\backslash T'$ and $p'$ be the restriction of $p$ to $F'$}.
			\end{unfolding}
			\begin{figure}[htbp]
				\center \input{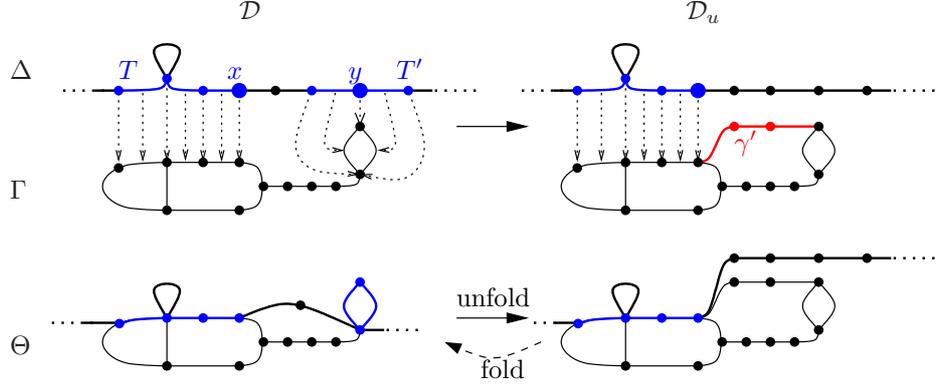}
				\caption{Unfolding of type \ref{unfold:remove_cc}.}
				\label{fig:remove_cc}
			\end{figure}
			\begin{unfolding}[See Figure \ref{fig:artificial_segment}] \label{unfold:artificial_segment} Suppose $v$ is an isolated vertex of $F$. Let $e$ be an edge adjacent to $v$. Let $\gamma$ be a new path of length two with both edge labeled by $\ell(e)$. Let $\Gamma'$ be obtained from $\Gamma \cup \gamma$ by gluing $\alpha(\gamma)$ to $p(v)$ \changed{and $F'=F$}. Define the map $p':F' \to \Gamma'$ as $p'(x) = p(x)$ if $x \neq v$ and $p'(v) = \omega(\gamma)$.
			\end{unfolding}
			\begin{figure}[htbp]
				\center \input{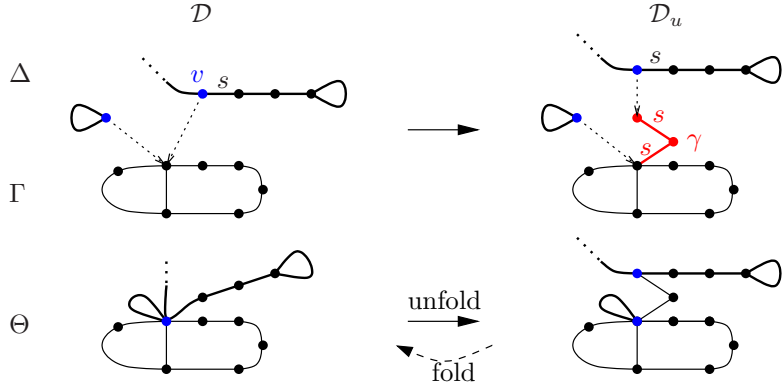}
				\caption{Unfolding of type \ref{unfold:artificial_segment}.}
				\label{fig:artificial_segment}
			\end{figure}
			\begin{remark} After an \changed{unfolding} of type \ref{unfold:artificial_segment} at an isolated vertex $v$ of $F$ there is no other vertex $v' \in F'$ such that $p'(v) = p'(v')$. Therefore, after using unfolding of type \ref{unfold:cc_vertex} to each connected component of $F$ followed by \changed{unfolding} of type \ref{unfold:artificial_segment} on each resulting vertex, \changed{condition \eqref{cond:Omega} is satisfied trivially. It follows in particular from the following lemma that $(\Dc',\emptyset)$ is tame if $(\Dc,\Omega)$ is tame for some $\Omega$.}
			\end{remark}
			\begin{lemma} \label{lemma:unfold_same_cpx}
				Suppose $F$ is a forest and $\Theta$ is connected. Then the primary complexity does not change when performing \changed{unfoldings} of type \ref{unfold:cc_vertex}, \ref{unfold:remove_cc} or \ref{unfold:artificial_segment}. Moreover, conditions \eqref{cond:pi1surj}, \eqref{cond:inj_edges_sppaths} and \eqref{cond:M_large} are unaffected by \changed{unfoldings} of type \ref{unfold:cc_vertex}, \ref{unfold:remove_cc} or \ref{unfold:artificial_segment}.
			\end{lemma}
			
			\begin{proof}
				The complexities depend only on $\Delta$ (which remains unchanged), $\betti(\Theta)$, $\Theta^f$\changed{ and $\Delta^f$}. We claim that $\betti(\Theta')$ is the same as $\betti(\Theta)$ after any of the unfolding moves above. Indeed, recall that $\chi(\Theta) = \chi(\Delta) + \chi(\Gamma) - \chi(F)$. Since $F$ is a forest and $\Theta$ is connected we have $\betti(F)=0$ and $\cc(\Theta)=1$ so the last formula can be rewritten as 
				\[\betti(\Theta) = \betti(\Delta) -\cc(\Delta)+ \betti(\Gamma) - \cc(\Gamma) + \cc(F) +1.\]
				The fact that $\betti(\Theta)=\betti(\Theta')$ now easily follows by examining each unfolding.
				
				Moreover, $\Theta'$ folds onto $\Theta$ so the folded graph $\Theta^f$ remains unchanged\changed{, the same is clearly true for $\Delta^f$}.
				
				It is clear that $\Theta'$ is connected provided $\Theta$ is, and since $\Theta$ is obtained from $\Theta'$ by a sequence of folds \changed{the fact that condition~\eqref{cond:pi1surj} is preserved follows from Lemma~\ref{lemma:fold_same_subgroup}. The fact that conditions~\eqref{cond:inj_edges_sppaths} and \eqref{cond:M_large} are preserved is obvious}.
			\end{proof}
			Lemma \ref{lemma:unfold_same_cpx} allows for the following interpretation of $c_2$:
			\begin{remark}[Interpretation of $c_2$] \label{remark:interpret_c_2} 
				Suppose $F$ is a forest, $\Theta$ is connected and $\Gamma$ is nonempty. Suppose $\Dc'$ is obtained from $\Dc$ by performing unfoldings of type \ref{unfold:remove_cc} repeatedly so that $F'$ has exactly one connected component in each connected component of $\Delta'$, and therefore $\Gamma'$ is connected. Using the computation in the proof of Lemma \ref{lemma:unfold_same_cpx} we get
				\[ c_2' = \betti(\Theta') + \chi(\Delta') = \betti(\Gamma') + \cc(F') - \cc(\Gamma') +1 = \betti(\Gamma')+\cc(\Delta').\]
				By Lemma \ref{lemma:unfold_same_cpx} unfolding does not change primary complexity so that $c_2 = \betti(\Gamma')+\cc(\Delta')$. Note that the reasoning fails if $\Gamma$ (and thereby $F$) is empty, but the fact that $c_2 = \betti(\Gamma')+\cc(\Delta')$ is still valid.
				
				Note moreover that $c_*(\Dc) = c_2 + \betti(\Delta) - |\Ec| = c_2 + \betti(\Delta \backslash \Ec)$. In the example of Figure \ref{fig:interpret_c_2} we have $c_2(\Dc) = c_2(\Dc') = \betti(\Gamma') + \cc(\Delta') = \changed{4+3}$ and $c_*(\Dc) = c_2 + \betti(\Delta\backslash \Ec) =\changed{4+3}+1$.
			\end{remark}
			\begin{figure}[htb]
				\center \input{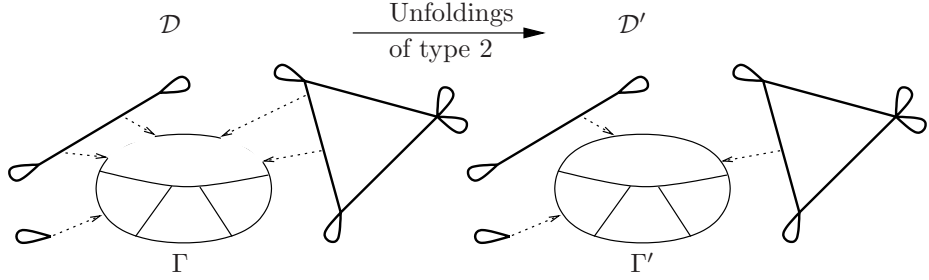}
				\caption{Interpretations of $c_2 = 4+3$ and $c_* = 4+3+1$.}
				\label{fig:interpret_c_2}
			\end{figure}
			
	\section{The folding sequence}
		\label{sec:folding_sequence}			
		
		\begin{lemma} \label{lemma:exist_minimalcpx_decomp} There is a tame marked decomposition of minimal complexity.
		\end{lemma}
		\begin{proof}	It is sufficient to show that the complexities $c_1, \ldots, c_7$ are non-negative integers for any tame marked decomposition. It is clear that $c_3, \ldots, c_7$ are always non-negative. Moreover, observe that $c_1 \geq c_2$. This is true since removing one inner edge from each special path of $\Delta$ yields a graph $\Delta'$ such that each connected component of $\Delta'$ contains a loop edge, and in particular $\chi(\Delta') \leq 0$. Thus 
			\begin{align*}
				c_1 &= \betti(\Theta) - |\Sppaths| \\
						&\geq \betti(\Theta) + \chi(\Delta') - |\Sppaths| = \betti(\Theta) + \chi(\Delta) = c_2.
			\end{align*}
			Finally if $(\Dc,\pairOmega)$ is a tame marked decomposition then $F$ is a forest by Lemma \ref{lemma:O3} \eqref{lemma:O3_forest} and $\Theta$ is connected, so the interpretation of $c_2$ following Lemma \ref{lemma:unfold_same_cpx} applies; in particular $c_2 \geq 0$.
		\end{proof}
		
		The following lemma makes essential use of the exponential bound in condition~\eqref{cond:M_large}.
		\begin{lemma}\label{lemma:Delta_folded}
			Let $(\Dc,\pairOmega)$ be a tame marked decomposition of minimal complexity. Then $\Delta$ is folded.
		\end{lemma}
		\begin{proof}
			Recall from Lemma \ref{lemma:O3} \eqref{lemma:O3_forest} that $F$ is a forest, so we can use \changed{unfoldings} of type \ref{unfold:cc_vertex}, \ref{unfold:remove_cc} and \ref{unfold:artificial_segment} in order to ensure that $F$ has exactly one vertex in each connected component of $\Delta$ and that the empty marking satisfies conditions \eqref{cond:Omega}, denote the new decomposition by $\Dc^u$. Thus $(\Dc^u,\emptyset)$ is tame. As unfoldings do not change the primary complexity by Lemma \ref{lemma:unfold_same_cpx} it follows that  $(\Dc^u,\emptyset)$ has the same primary complexity as $(\Dc,\pairOmega)$. Note that $\Delta^u=\Delta$, in particular $\Delta^u$ is folded if and only if $\Delta$ is folded.
			
			Suppose $\Delta^u=\Delta$ is not folded. Then there are two edges $e,e'$ starting at $v=\alpha(e) = \alpha(e')$ with $\ell(e) = \ell(e')$. In general $e$ and $e'$ cannot be identified without violating the conditions for $\Delta$ to be a special graph. We will distinguish 4 cases, in each of them we replace $\Delta^u$ by some new special graph $\Delta'$ \changed{that is obtained from $\Delta$ by first unfolding $\Delta$ to some graph $\tilde \Delta$ and then folding $\tilde\Delta$ onto $\Delta'$; note that we do not mention $\tilde \Delta$ if we do not unfold, i.e. if $\Delta=\tilde\Delta$.  For each vertex $u\in F^u$ choose a vertex $u'\in\Delta'$ by first choosing some lift $\tilde u$ of $u$ in $\tilde\Delta$ and then mapping $\tilde u$ to $\Delta'$. We then put $F'=\{u'\,|\,u\in F^u\}$ and } $p':F'\to\Gamma'=\Gamma^u$ by $p'(u')=p^u(u)$. We then define $\Dc'$ to be the decomposition obtained from $\Dc^u$ by replacing $\Delta=\Delta^u$ with $\Delta'$, $F^u$ by $F'$ and $p^u$ by $p'$. In all four cases it will be obvious that  $(\Dc',\emptyset)$ is a decomposition satisfying \eqref{cond:Omega}, \eqref{cond:pi1surj} and \eqref{cond:inj_edges_sppaths}.
			
In the first 3 cases $\Delta '$ is obtained from $\Delta$ by replacing the component of $\Delta$ that contains $v$ by a new graph obtained from that component by a sequence of unfolds and folds. This implies in particular that $\Theta^ f$ and $\Delta^ f$ and therefore $c_3$ and $c_4$ are preserved. In those 3  cases it is immediate that $c_*$ does not increase. As moreover the matrix $M$ is preserved this implies that \eqref{cond:M_large} holds trivially, i.e.\ that $(\Dc',\emptyset)$ is tame. We will obtain a contradiction to the minimality of the primary complexity. As $c_3$ and $c_4$ are unchanged we only need to discuss $c_1$, $c_2$ and $c_5$. 

In case~4 the argument is of a different nature, in fact the argument in that case explains the relevance of condition \eqref{cond:M_large}, i.e.\ explains where the exponential bound in the main theorem comes from.

			\setcounter{case}{0}
			\begin{case}
			If both $e$ and $e'$ are loop edges then identifying them reduces $\betti(\Theta)$ and leaves the number of special paths unchanged. Thus $c_1$ decreases contradicting the minimality of the complexity.
			\end{case}
			
			\begin{case} \label{case:fold_Delta_ENE}
			Suppose that $e$ is a loop edge and that $e'$ is not, as in Figure~\ref{fig:fold_Delta_ENE}. Therefore $e'$ is the first edge of some special path $\delta$. Possibly after unfolding $\Delta$ to \changed{$\tilde \Delta$} along the extremal edges of $\delta$, see Figure~\ref{fig:unfold_sppath}, we can assume that $\delta$ intersects other special paths only at its endpoints. 			
			\begin{figure}[htb]
				\center \input{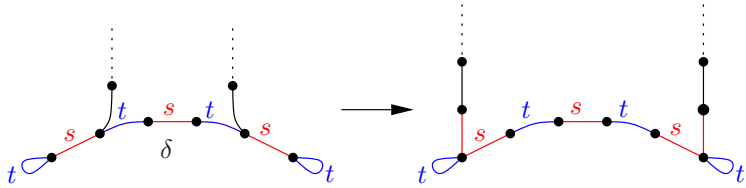}
				\caption{Unfolding the special path $\delta$ at its extremal edges.}
				\label{fig:unfold_sppath}
			\end{figure}

			\changed{Note that $\tilde\Delta$ has up to two more edges than $\Delta$} and therefore $c_5$  increases by up to two. Moreover the complexities $c_1$ and $c_2$ are unchanged. Let $\{s,t\}$ be the type of $\delta$, where $\ell(e) = \ell(e') = s$. Since $v$ is the initial vertex of $\delta$ and the first edge of $\delta$ has label $s$, there is a loop edge $f$ based at $v$ with $\ell(f)=t$. Thus the whole path $\delta$ can be folded on the two loop edges $e$ and $f$. \changed{Call the resulting graph $\Delta'$.}

			\begin{figure}[htb]
				\center \input{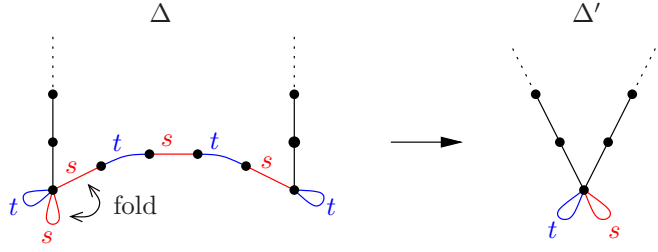}
				\caption{Lemma \ref{lemma:Delta_folded} Case \ref{case:fold_Delta_ENE}: folding a loop edge with a non-loop edge.}
				\label{fig:fold_Delta_ENE}
			\end{figure}
			
 If $\alpha(\delta) \neq \omega(\delta)$ then after folding $\delta$, there are two loops with the same label based at $v$, which we fold. Whether or not $\alpha(\delta) = \omega(\delta)$ the new graph $\Delta'$ obtained after folding has one less special path and $\betti(\Delta') \le\betti(\Delta)-1$. Therefore $c_1=c_1'$ and $c_2 = c_2'$. Lastly, $c_5$ decreases by at least $5$ \changed{ when going from $\tilde\Delta$ to $\Delta'$}, so $\Dc'$ has a smaller complexity than $\Dc$ as this decrease by at least $5$ outweighs the initial increase by at most $2$. This contradicts the minimality of the complexity of $(\Dc,\pairOmega)$.
			
			\end{case}
			
			We assume from now on that neither $e$ nor $e'$ is a loop edge. If $e$ and $e'$ are extremal edges of distinct special paths \changed{and do not lie on a common (necessarily closed) special path} then identifying $e$ and $e'$ is allowed and this identification does not change complexities $c_1$ and $c_2$ and reduces the number of edges of $\Delta$ and therefore $c_5$, a contradiction. Thus we consider two cases. On the one hand we treat the case where both $e$ and $e'$ are the second edges of unique special paths $\delta$ and $\delta'$ with $\alpha(\delta) = \alpha(\delta')$ and $\delta \neq {\delta'}\inv$. On the other hand, we treat the case where $e$ and $e'$ are the two extremal edges of a simple closed special path $\delta$.
			
			\begin{case} \label{case:fold_Delta_NENE}
			 Suppose both $e$ and $e'$ are the second edges of unique special paths $\delta$ and $\delta'$ with $\alpha(\delta) = \alpha(\delta')$ and $\delta \neq {\delta'}\inv$. Since $\delta$ and $\delta'$ have a common initial edge and have second edge $e$ and $e'$ with the same labeling, $\delta$ and $\delta'$ have the same type $\{s,t\}$. Therefore it is possible to fold $\delta$ and $\delta'$ together. If $\omega(\delta) \neq \omega(\delta')$ then also fold the two loop edges with label $s$ or $t$ at the end of $\delta$ and $\delta'$ together. This is illustrated in Figure~\ref{fig:fold_Delta_NENE}. Whether or not $\delta$ and $\delta'$ are closed or $\omega(\delta)=\omega(\delta')$ the new graph $\Delta'$ obtained after folding has one less special path and Betti number one less than $\Delta$. 
It follows that  the complexities $c_1$ and $c_2$ do not increase. Moreover, the number of edges of $\Delta'$ is strictly smaller than that of $\Delta$, so that $c_5' < c_5$, contradicting the minimality of the complexity.
			\begin{figure}[htb]
				\center \input{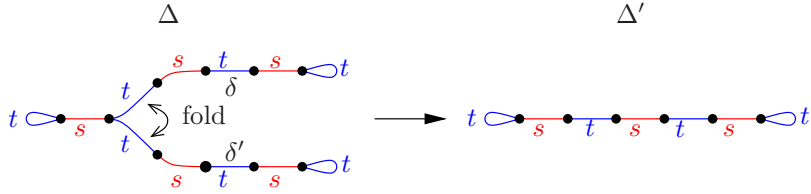}
				\caption{Lemma \ref{lemma:Delta_folded} Case \ref{case:fold_Delta_NENE}: folding two non-loop edges.}
				\label{fig:fold_Delta_NENE}
			\end{figure}
			\end{case}
			
			\begin{case} \label{case:fold_Delta_half}
			Suppose $e$ and $e'$ are the two extremal edges of a simple closed special path $\delta$. Since the labels of $e$ and $e'$ are the same and $\delta$ is of length $m_{st}-1$ it follows that $m_{st}$ is even, in particular $\delta$ is of odd length. Let $M'$ be the Coxeter matrix obtained by replacing $m_{st}$ and $m_{ts}$ by $m_{st}'=m_{st}/2$ in $M$ and leaving all other entries unchanged. The aim is to construct an $M'$-special graph $\Delta'$. In order to satisfy condition \eqref{cond:Delta_relator} we need to modify each special path $\delta_i$ of type $\{s,t\}$. Thus for each special path $\delta_*$ of type $\{s,t\}$ we proceed as follows (see Figure \ref{fig:fold_Delta_half}). Write $\delta_*$ as $e_0,e_1, \ldots,e_n$. As $m_{st}$ is even, $\delta_*$ has an odd number of edges so $\ell(e_i) = \ell(e_{n-i})$. For each $i \neq \frac{n}{2}$ identify $e_i$ with $e_{n-i}\inv$. Moreover, if before identification $\alpha(\delta_*) \neq \omega(\delta_*)$, then after identification there are two loop edges $e,f$ that have the same label based at $\alpha(\delta_*)$. Identify $e$ and $f$ and denote the resulting loop edge by $e'$. Notice that after this identification, $e_{\frac{n}{2}}$ is a loop edge, and letting $\delta_*'$ be the path $e_0,e_1,\cdots,e_{\frac{n}{2}-1}$ we have that $\ell(\delta_*' . e_{\frac{n}{2}} . {\delta_*'}\inv . e')$ is either $(st)^{m_{st}'}$ or $(ts)^{m_{ts}'}$. We replace the $M$-special path $\delta_*$ by the $M'$-special path $\delta_*'$.
			\begin{figure}
				\center \input{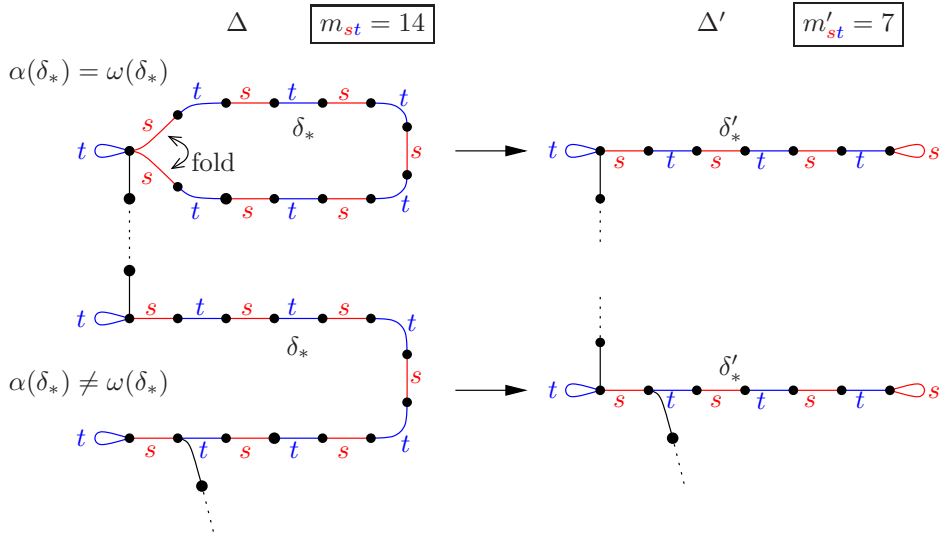}
				\caption{Lemma \ref{lemma:Delta_folded} Case \ref{case:fold_Delta_half}: folding $\Delta$.}
				\label{fig:fold_Delta_half}
			\end{figure}
			
			Let $\Delta'$ denote the graph obtained by these identifications. Clearly $\Delta'$ is an $M'$-special graph. Observe that $\betti(\Delta) = \betti(\Delta')$ and that the number of special paths is unchanged, so $c_1' = c_1$ and $c_2' = c_2$. Since $\Delta$ maps onto $\Delta'$ one also has that $c_3' = c_3$ and $c_4' = c_4$. Moreover, it is clear that $c_5' = |E\Delta'| < |E\Delta| = c_5$. Thus the new tame marked decomposition $(\Dc',\emptyset)$ has smaller complexity than $(\Dc^u,\emptyset)$. Moreover $c_*' \leq c_*-1$ because $\Delta'$ has at least one more loop edge than $\Delta$, so that 
			\[m'_{st} = m_{st}/2 \geq \Konst.2^{c_*}/2 \geq \Konst.2^{c_*'}\] and the decomposition $\Dc'$ satisfies condition \eqref{cond:M_large}, contradicting the minimality of the complexity. \qedhere
			\end{case}
		\end{proof}
		
		\begin{remark} \label{rem:sppath_contractible} If $\Delta$ is folded, then each special path is embedded. Arguing by contradiction, we suppose $\delta$ is a special path of type $\{s,t\}$ that is not embedded. Recall that $\alpha(\delta) = \omega(\delta)$ by condition \eqref{cond:sppaths_emb_or_closed}. If $m_{st}$ is even then the first few edges of $\delta$ can be folded with the last few edges of $\delta$, so $\Delta$ is not folded. If $m_{st}$ is odd, then there are two loop edges based at $\alpha(\delta) = \omega(\delta)$ with labels $s$ and $t$ respectively. But the (non-loop) edges of $\delta$ also have label $s$ or $t$, so $\Delta$ is not folded in this case either.
		\end{remark}
		
		\begin{lemma} \label{lemma:delta_12}
			Let $(\Dc,\pairOmega)$ be a tame marked decomposition of minimal complexity. Then the following conditions hold:
			\begin{customenum}{$\mathbf{\bar\Delta}$}
				\item Let $x,y$ be distinct vertices in $\Delta$ lying on special paths of the same type. Then $x$ and $y$ do not project \changed{to} the same vertex of $\bar \Delta$. \label{cond:two_vert_same_type}
				\item If two edges in $\bar \Delta$ are adjacent and one of them is a loop edge then they have different labels. \label{cond:no_loop_on_special_path}
			\end{customenum}
		\end{lemma}
		\begin{proof} The proofs in all cases are by contradiction to minimality, i.e.\ we show that the complexity was not minimal provided the conclusion fails. In all cases we construct decompositions of smaller complexity and it is always immediate from the construction that 
		for any closed path in the original graph $\Theta$ there is a path  with the same labeling in the new graph $\Theta$ (fixing some appropriate basepoint in each graph).
		Thus property \eqref{cond:pi1surj} always holds for the constructed decomposition and we ignore to mention it in the remainder of this proof.
		
		We first prove that \eqref{cond:two_vert_same_type} holds.
			The argument is illustrated in Figure \ref{fig:two_vert_same_type}.
			\begin{figure}
				\center\input{two_vert_same_type.pstex_t}
				\caption{Lemma \ref{lemma:delta_12} \eqref{cond:two_vert_same_type}: constructing $\Dc'$ of smaller complexity.}
				\label{fig:two_vert_same_type}
			\end{figure}
			Suppose there are two vertices $x$ and $y$ lying on special paths of the same type that project to the same point in $\Theta$, or in other words $x$ and $y$ are in $F$ and $p(x) = p(y)$. Let $F_x$ and $F_y$ denote the connected components of $F$ containing $x$ and $y$ respectively. Similarly let $\Delta_x$ and $\Delta_y$ denote the connected components of $\Delta$ containing $x$ and $y$ respectively. We deal with two cases, depending on whether $F_x=F_y$ or not. 
			\setcounter{case}{0}
			\begin{case} \label{case:same_cc} Suppose that $F_x=F_y$. Let $\gamma$ be a reduced path in $F$ from $x$ to $y$. By Lemma~\ref{lemma:O3}~\eqref{lemma:O3_small} the path $\gamma$ does not contain all inner edges of a special path. By Lemma~\ref{lemma:Delta_folded} any two special paths of the same type do not intersect in $\Delta$, and any special path is embedded in $\Delta$. Therefore $\gamma$ is a subpath of a special path $\delta$. Let $\hat \gamma$ be the (unique) lift of $\bar \gamma$ in $\Gamma$, or in other words $\hat \gamma = p(\gamma)$. Since $\bar \gamma$ is a closed, so is its lift $\hat \gamma$.
				
				We unfold $\Dc$ to a new decomposition $\Dc^u$ as follows. Unfold each connected component of $F$ to a vertex using unfolding of type \ref{unfold:cc_vertex}, where $x$ is the chosen vertex for its connected component. Using unfolding of type \ref{unfold:remove_cc} we can assume $x$ is the only vertex of $F$ in $\Delta_x$, and in particular $x$ is the only vertex of $F$ in the special path $\delta$. We also apply unfolding of type \ref{unfold:artificial_segment} to each vertex of $F$ except $x$, so that the empty marking satisfies condition  \eqref{cond:Omega}. Unfold the special path $\delta$ at its boundary edges if applicable (see Figure \ref{fig:unfold_sppath}). Note that doing this may increase $c_5$. 
				
				\changed{Finally, we let $\Dc'$ be obtained from $\Dc^u$ by the following surgery: collapse $\hat \gamma\subset \Gamma^u$ to a vertex $z$, yielding $\Gamma'$; collapse $\delta\subset \Delta^u$ to a vertex $x'$  and if both boundary loop edges of $\delta$ had the same label, relabel one of them with the other label in the type of $\delta$, call the resulting graph $\Delta'$. Let $F'=(F^u\backslash \{x\})\cup\{x'\}$ and define $p':F'\to\Gamma'$ by $p'(x') = z$ and $p'(v)=p^u(v)$ otherwise}. Note that $\betti(\Theta') = \betti(\Theta) -1$ but $\Delta'$ has one less special path than \changed{$\Delta$}. Thus $c_1' = c_1$, $c_2' = c_2 -1$ and $c_* = c_*-1$. In conclusion $(\Dc',\emptyset)$ is a tame marked decomposition, contradicting the minimality of the complexity of $(\Dc,\pairOmega)$.
			\end{case}
			
			From now on we assume that $F_x\neq F_y$. Let $\delta_x$ and $\delta_y$ be special paths of the same type containing $x$ and $y$ respectively. Unfold each connected component of $F$ to a vertex using unfolding of type \ref{unfold:cc_vertex}, where $x$ and $y$ are the chosen vertices for their connected components. Using unfolding of type \ref{unfold:remove_cc} we can assume $x$ and $y$ are the only vertices of $F$ in $\Delta_x \cup \Delta_y$, in particular $x$ and $y$ are the only vertices of $F$ in $\delta_x \cup \delta_y$. We also apply unfolding of type \ref{unfold:artificial_segment} to each vertex of $F$ except $x$ and $y$, so that the empty marking satisfies condition \eqref{cond:Omega}. Unfold the special paths $\delta_x$ and $\delta_y$ at their boundary edges. Let $\Dc^u$ denote the unfolded decomposition. We split this case into two. 
			
			\begin{case} \label{case:dx_eq_dy}If $\delta_x = \delta_y=:\delta$ do the following surgery: \changed{collapse $\delta$ to a point $x'$ and if both boundary loop edges of $\delta$ had the same label, relabel one of them with the other label in the type of $\delta$, call the resulting graph $\Delta'$. Put $F'=(F^u\backslash \{x,y\})\cup\{x'\}$ and define $p':F'\to \Gamma':=\Gamma^u$ as $p'(x') = p(x)$ and $p'(v)=p^u(v)$ for $v\in F^u\backslash \{x,y\}$}. Note that $F'$ has one less connected component than $F^u$ so that \changed{$\betti(\Theta') = \betti(\Theta^u) -1= \betti(\Theta) -1$}. Moreover $\Delta'$ has one less special path than \changed{$\Delta^u=\Delta$} so $c_1' = c_1$, $c_2' = c_2 -1$ and $c_* = c_*-1$. Again $(\Dc',\emptyset)$ is a tame marked decomposition, contradicting the minimality of the complexity of $(\Dc,\pairOmega)$.
			\end{case}
			\begin{case} \label{case:dx_neq_dy} Suppose now that $\delta_x$ is distinct from $\delta_y$. Note that $\delta_x$ and $\delta_y$ do not intersect as they are of the same type and $\Delta$ is folded by Lemma \ref{lemma:Delta_folded}. We proceed as follows. Collapse $\delta_x \cup \delta_y$ to a vertex \changed{$x'$ and replace the four boundary loop edges of $\delta_x$ and $\delta_y$ by two loop edges labeled by the two labels in the type of $\delta_x$. Call the resulting graph $\Delta'$}. \changed{Let $\Gamma':=\Gamma^u$ and define $F'$ and $p'$ as in case \ref{case:dx_eq_dy}}. Now $F'$ has one less connected component than \changed{$F^u$}, and $\Delta'$ has two less special paths and two less loop edges \changed{than $\Delta^u$}. Moreover two contractible subsets got identified to one vertex so \changed{$\chi(\Delta') = \chi(\Delta^u)+2-1=\chi(\Delta^u) + 1=\chi(\Delta)+1$}. Indeed, a special path is contractible in $\Delta^u$ by \changed{the remark preceding this lemma.} 
Therefore \changed{$\betti(\Theta') = \betti(\Theta^u) - 2=\betti(\Theta)-2$}. Plugging this information into the complexities we get $c_1' = c_1$ and $c_2' = c_2 -1$. Moreover $c_*' \leq c_*$, so $(\Dc',\emptyset)$ is a tame marked decomposition, contradicting the minimality of the complexity of $(\Dc,\pairOmega)$.
			\end{case}
			
\smallskip 	We now show that condition 		 Condition  \eqref{cond:no_loop_on_special_path} holds.

			\setcounter{case}{0}
			\begin{case} Suppose there is a loop edge $e$ in $\bar \Delta$ that lifts to a non-loop edge $\tilde e$ of $\Delta$. Then the two vertices $\alpha(\tilde e)$ and $\omega(\tilde e)$ of $\Delta$ are contained in a special path and are distinct in $\Delta$, so condition \eqref{cond:two_vert_same_type} does not hold for $\Dc$. We have already shown that this contradicts the minimality of the complexity of $(\Dc,\pairOmega)$.
			\end{case}
			
			From now on for a loop edge $e$ in $\bar \Delta$ we denote by $\tilde e$ its (unique) lift as a loop edge in $\Delta$.
			
			\begin{case} \label{case:two_loop_edges} Suppose that there are two loop edges $e \neq e'$ in $\bar \Delta$ that have the same label and same base vertex. Let $\tilde e$ and $\tilde e'$ be their lifts in $\Delta$ with basepoints $x$ and $y$ respectively. Notice that $x \neq y$ as otherwise $\Delta$ would not be folded, contradicting Lemma \ref{lemma:Delta_folded}. Moreover, $x$ and $y$ do not lie in the same connected component of $F$ by Lemma~\ref{lemma:O3}~\eqref{lemma:O3_small}. Use unfolding of type \ref{unfold:cc_vertex} to replace $F$ by a set of vertices choosing $x$ and $y$ in their respective connected components. Also use unfolding of type \ref{unfold:artificial_segment} for each vertex of $F$ other than $x$ and $y$. \changed{Call the resulting decomposition $\Dc^u$. Now construct a new decomposition $\Dc'$ by} the following surgery, depicted in Figure \ref{fig:two_loop_edges}. Identify $x$ and $y$ to a new vertex $x'$ and identify the two edges $\tilde e$ and $\tilde e'$, yielding \changed{$ \Delta'$. Let $\Gamma':=\Gamma^u=\Gamma$, $F':=(F^u\backslash\{x,y\})\cup\{x'\}$ and define $p':F'\to\Gamma'$ as $p'(x') = p(x)$ and $p'(v)=p^u(v)$ for $v\in F^u\backslash\{x,y\}$}. Now \changed{$\chi(\Delta') = \chi(\Delta^u)=\chi(\Delta)$} but $F'$ has one connected component less than $F^u$. Therefore \changed{$\betti(\Theta') = \betti(\Theta^u) - 1=\betti(\Theta)-1$} and $c_1' = c_1 - 1$. Moreover $c_*' \leq c_*$, so $(\Dc',\emptyset)$ is a tame marked decomposition, contradicting the minimality of the complexity of $(\Dc,\pairOmega)$.
			\begin{figure}
				\center \input{two_loop_edges.pstex_t}
				\caption{Lemma \ref{lemma:delta_12} \eqref{cond:no_loop_on_special_path} Case \ref{case:two_loop_edges}: Folding two loop edges.}
				\label{fig:two_loop_edges}
			\end{figure}
			\end{case}
			
			\begin{case} \label{case:loop_non_loop} Now suppose there is a loop edge $e$ and a non-loop edge $e'$ in $\bar \Delta$ having the same initial vertex. Let $\tilde e$ and $\tilde e'$ be lifts of $e$ and $e'$ respectively in $\Delta$. Let $x$ be the basevertex of $\tilde e$ and $y = \alpha(\tilde e')$. Since $\Delta$ is folded by Lemma \ref{lemma:Delta_folded} the vertices $x$ and $y$ are distinct. If $x$ and $y$ lie on the same special path then $\Dc$ does not satisfy condition \eqref{cond:two_vert_same_type} which is already ruled out. So without loss of generality we can assume that there is a special path $\delta$ containing $\tilde e'$ but not $x$. Moreover since $\tilde e$ and $\tilde e'$ have the same label, the label of $\tilde e$ is contained in the type of $\delta$. Lemma~\ref{lemma:O3}~\eqref{lemma:O3_small} implies that $x$ and $y$ are not in the same connected component of $F$. Apply unfolding of type \ref{unfold:cc_vertex} on each connected component of $F$ choosing $x$ and $y$ in their respective connected components. Also apply unfolding of type \ref{unfold:artificial_segment} to each vertex of $F$ other than $x$ and $y$. Unfold the extremal edges of $\delta$ in $\Delta$ if applicable, \changed{ call the resulting decomposition $\Dc^u$. Then perform the following surgery yielding a new decomposition $\Dc'$, see Figure~\ref{fig:loop_non_loop}}. 

	\begin{figure}
				\center \input{loop_non_loop.pstex_t}
				\caption{Lemma \ref{lemma:delta_12} \eqref{cond:no_loop_on_special_path} Case \ref{case:loop_non_loop}: Doing surgery.}
				\label{fig:loop_non_loop}
			\end{figure}

\changed{Collapse $\delta \cup \{x\}\subset \Delta^u$ to a vertex $x'$. If both boundary loop edges of $\delta$ had the same label, relabel one of them with the other label in the type of $\delta$.} Now $\tilde e$ can be folded onto one of these two boundary loop edges. \changed{Denote the resulting graph by $\Delta'$.} \changed{Put $F'=(F^u \backslash \{x,y\})\cup\{x'\}$, $\Gamma':=\Gamma^u=\Gamma$  and define $p':F'\to\Gamma'$ as $p'(x')=p(x)$ and $p'(v)=p^u(v)$ otherwise}. Two disjoint contractible subsets of \changed{$\Delta^u$} were collapsed to a point of $\Delta'$ and $\Delta'$ has one loop edge less than \changed{$\Delta^u$}, so that $\chi(\Delta') = \chi(\changed{\Delta^u})=\chi(\Delta)$. Moreover, $\Delta'$ has one special path less than $\changed{\Delta^u}$ and $F'$ has one connected component less than \changed{$F^u$}. In conclusion $c_1' = c_1$, $c_2' = c_2-1$ and $c_*' \leq c_*$. Thus $(\Dc',\emptyset)$ is a tame marked decomposition, contradicting the minimality of the complexity of $(\Dc,\pairOmega)$. \qedhere

			\end{case}
		\end{proof}

			
			

Assume that $\gamma$ is a reduced simple path in $\Delta$ of length at most $4$. As all special paths are of length at least 5 it is clear that either $\gamma$ is a subpath of some special path or $\gamma$ is the concatenation of two path $\gamma_1$ and $\gamma_2$ that are both subpaths of special paths. We then say that $\gamma$ \emph{turns} at $v=\omega(\gamma_1)=\alpha(\gamma_2)$. Note that $v$ is in distance at most one of some loop vertex. 

\begin{lemma}\label{intersection}  Let $(\Dc,\pairOmega)$ be a tame marked decomposition and $\bar\gamma$ be reduced path in $\bar\Delta$ of length $k$ that has two distinct lifts $\gamma_1=e_1,\ldots ,e_k$ and $\gamma_2=f_1,\ldots ,f_k$. Then $k\le 3$. Moreover if $k=3$ then one of the following holds:
\begin{enumerate}
\item  $\gamma_1$ turns at  $\alpha(e_2)$ and $\gamma_2$ turns at $\alpha(f_3)$.
\item $\gamma_1$ turns at $\alpha(e_3)$ and $\gamma_2$ turns at $\alpha(f_2)$. 
\end{enumerate}
\end{lemma}

\begin{proof} If $k\ge 4$ or $k=3$ and neither (1) nor (2) hold then the above discussion implies that for some $i$ both $e_i,e_{i+1}$ and $f_i,f_{i+1}$ are subpath of special paths $\delta$ and $\delta'$ which must be of the same type. This contradicts condition  \eqref{cond:two_vert_same_type}.
\end{proof}

		\changedm{\begin{corollary} \label{cor:dist} Let $(\Dc,\pairOmega)$ be a tame marked decomposition of minimal complexity. Let $x \neq y \in VF$ be such that $p(x)=p(y)$, and let $e \in E\Delta$ be a non-loop edge such that $\alpha(e) = x$. 
		\begin{enumerate}
			\item If $\omega(e)$ is the basepoint of a loop edge, then $e \in \hOmega_3$. \label{cor:dist_no_loop_edge}
			\item If $e'$ is a non-loop edge of $\Delta$ such that $\alpha(e') = y$ and such that $\ell(e) = \ell(e')$, then $e,e' \in \hOmega_3$. \label{cor:dist_at_most_3}
		\end{enumerate}
		\end{corollary}
		
		
		
		\begin{proof} We argue by contradiction and suppose either \eqref{cor:dist_no_loop_edge} or \eqref{cor:dist_at_most_3} fails. Thus we assume that $e \notin \hOmega_3$ as the roles of $e$ and $e'$ in \eqref{cor:dist_at_most_3} are interchangeable. By condition \eqref{cond:Omega_vertex} $x,y \in \hOmega_3$. Let $\gamma$ and $\gamma'$ be shortest paths from $\hOmega$ to $x$ and $y$ respectively. Observe that $\gamma$ has length exactly $3$, as otherwise $e$ would lie in $\hOmega_3$. Write $\gamma\equiv e_1,e_2,e_3$ and $\gamma' \equiv e_1',e_2',e_3'$. Note that $\bar\gamma=\bar\gamma'$ by Lemma~\ref{homotopyOmega3} (3), in particular the conclusion of Lemma~\ref{intersection} applies to $\gamma$ and $\gamma'$.
			
			\changedm{\eqref{cor:dist_no_loop_edge} Since $\gamma$ either turns at $\alpha(e_2)$ or $\alpha(e_3)$, it follows that $\alpha(e_2)$ or $\alpha(e_3)$ is at distance at most 1 of some basepoint of a loop edge of $\Delta$. If $\omega(e)$ is also the basepoint of a loop edge then there are two basepoints of loop edges that are at distance at most 4 apart, a contradiction to the assumption that special path are of length at least~5.}
			
			\eqref{cor:dist_at_most_3} Note that $e_1,e_2,e_3,e$ and $e_1',e_2',e_3',e'$ are reduced paths in $\Delta$ and that both $e_3,e$ and $e_3',e'$ are subpath of special paths $\delta$ and $\delta'$ as they cannot turn at $\alpha(e)$ and $\alpha(e')$, respectively. Therefore $\delta$ and $\delta'$ must be of the same type. This clearly contradicts condition \eqref{cond:two_vert_same_type}, see Figure \ref{fig:stay_in_O3}.
			\begin{figure}[htb]
				\center\input{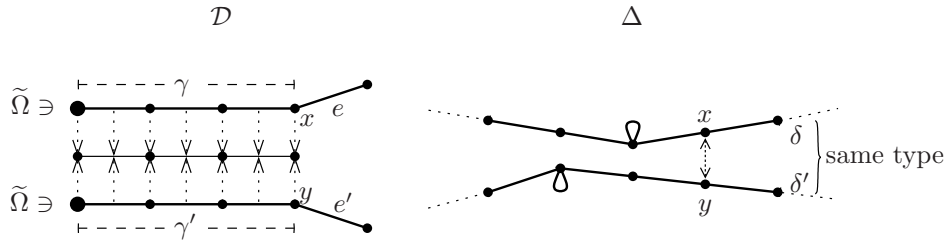}
				\caption{Corollary~\ref{cor:dist}~\eqref{cor:dist_at_most_3} : a contradiction to condition \eqref{cond:two_vert_same_type}.}
				\label{fig:stay_in_O3}
			\end{figure}
		\end{proof}}
		
		
		\begin{lemma} \label{lemma:EF_EOmega} Suppose that $(\Dc,\pairOmega)$ is a tame marked decomposition of minimal complexity. Then $EF = E\hOmega_3$ and $E\bar F = E\mOmega_3$.
		\end{lemma}
		\begin{proof} By condition \eqref{cond:Omega_edge} $EF \subset E\hOmega_3$. First we show that $EF = E\hOmega_3$. Suppose by contradiction that there is an edge $e \in E\hOmega_3 \backslash EF$. Put $M' = M$, $\Delta' = \Delta$, $\Theta' = \Theta$ and $\mOmega' = \mOmega \subset \Theta$. Put $\Gamma' = (\Gamma \cup f) / \sim$ for some new edge $f$ where $\alpha(f) \sim p(\alpha(e))$ if $\alpha(e) \in F$ and $\omega(f) \sim p(\omega(e))$ if $\omega(e) \in F$. In all cases let $F' = F \cup e$ and define $p'|_F = p$ and $p(e) = f$. This does not change $\Delta$, $\Theta$ or $\Theta^f$, so that the primary complexity does not change. Note moreover that $|E\hOmega_3' \backslash EF'| = |E\hOmega_3 \backslash EF|-1$ and that $|E\Theta| = |E\Theta|$ so that $c_7' = c_7 - 1$. Lastly note that $c_*' = c_*$ so that $(\Dc',\pairOmega')$ satisfies condition \eqref{cond:M_large}. This is a contradiction to the minimality of the complexity of $(\Dc,\pairOmega)$.
		
			It remains to show that $E\mOmega_3 = E \bar F$. Since we know that $E\hOmega_3 = EF$ it is enough to show that $\hOmega_3$ is the preimage of $\mOmega_3$ in $\Delta \backslash \Ec$ which is the content of Lemma~\ref{homotopyOmega3}~(4).
		\end{proof}
		
		\begin{lemma} \label{lemma:IIIA} Let $(\Dc,\pairOmega)$ be a tame marked decomposition of minimal complexity. Suppose $e$ and $e'$ are distinct edges of $\bar \Delta$ having the same initial vertex and same label. Then $\omega(e) \neq \omega(e')$.
		\end{lemma}
		\begin{proof} Assume that $\omega(e)=\omega(e')$. Let $\tilde e$ and $\tilde e'$ be lifts in $\Delta$ of $e$ and $e'$ respectively.  Since $\Delta$ is folded by Lemma~\ref{lemma:Delta_folded} the origin of $\tilde e$ and $\tilde e'$ are different, and thus both $\tilde e$ and $\tilde e'$ must lie in $\hOmega_3$ by Corollary~\ref{cor:dist}~\eqref{cor:dist_at_most_3}. The goal is to construct a tame marked decomposition $(\Dc',\pairOmega')$ of smaller complexity. 
				
				By Lemma \ref{lemma:EF_EOmega} $e$ and $e'$ lie in $\bar F$ so that there are lifts $\hat e,\hat e'$ of $e,e'$ in $\Gamma$. Note that $\omega(\hat e)=\omega(\hat e')$. Let $\Dc^u$ be the decomposition obtained from $\Dc$ by performing unfolding of type \ref{unfold:cc_vertex} on each connected component of $F$, followed by performing on each resulting vertex an unfolding of type \ref{unfold:artificial_segment}. Since $EF^u = \emptyset$ has no edges and since $p^u$ is injective on $VF^u$ the marked decomposition $(\Dc^u,\emptyset)$ satisfies condition \eqref{cond:Omega}. In fact $(\Dc^u,\emptyset)$ is a tame marked decomposition with the same primary complexity as $(\Dc,\pairOmega)$ by Lemma \ref{lemma:unfold_same_cpx}. We use the same notation for $\hat e,\hat e'$ and their obvious images in $\Gamma^u$.
				
				We define a new decomposition $\Dc'$ as follows: let $\Gamma':=\Gamma^u/\hat e \sim \hat e'$. This is possible since $\ell(\hat e) = \ell(\hat e')$ by hypothesis. Put $\Delta':= \Delta^u$, $F':=F^u$ and $p':=\phi \circ p^u$ where $\phi:\Gamma^u \to \Gamma'$ is the folding map. Notice that $(\Dc',\emptyset)$ is a tame marked decomposition such that $\betti(\Theta') = \betti(\Theta^u)-1$ so that $c_1' = c_1^u-1 = c_1 - 1$ contradicting the minimality of the complexity of the tame marked decomposition $(\Dc,\pairOmega)$.
			\begin{figure}
				\begin{center}\input{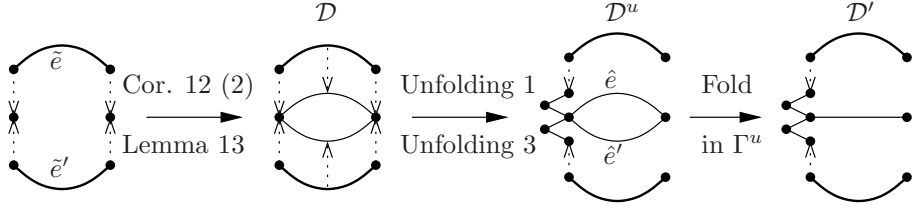}\end{center}
				\caption{Lemma \ref{lemma:IIIA}: contradicting the minimality of the complexity of $\Dc$}
				\label{fig:step2_IIIA}
			\end{figure}
		\end{proof}
		
		The following shows that in order to admit some marking with \eqref{cond:Omega}, we only need to verify \eqref{cond:Omega_vertex}, \eqref{cond:Omega_edge} and \eqref{cond:Omega_complexity}.
		\begin{lemma} \label{lemma:Omega_path} Let $(\Dc,\pairOmega)$ be a marked decomposition satisfying conditions \eqref{cond:Omega_vertex}, \eqref{cond:Omega_edge} and \eqref{cond:Omega_complexity}. Then there is another marking $\hOmega'$ such that $(\Dc,\pairOmega')$ satisfies \eqref{cond:Omega_vertex}, \eqref{cond:Omega_edge}, \eqref{cond:Omega_complexity} and \eqref{cond:Omega_path}.
		\end{lemma}
		\begin{proof} Note that there are only finitely many markings. Thus there exists a marking $\hOmega'$ such that $\mOmega'$ is of minimal Euler characteristic among all markings satisfying conditions \eqref{cond:Omega_vertex}, \eqref{cond:Omega_edge} and \eqref{cond:Omega_complexity}. We show that $(\Dc,\pairOmega')$ also satisfies condition \eqref{cond:Omega_path}. Suppose by contradiction that there exists a reduced path $\gamma$ of length at most $8$ in $\overline{\Delta \backslash \Ec}$ having both endpoints in $\mOmega'$ but such that $\gamma$ is not entirely contained in $\mOmega'$. Clearly $\gamma$ has more edges outside $\mOmega'$ than it has vertices outside $\mOmega'$. Therefore the graph $\mOmega'' = \mOmega' \cup \gamma$ has Euler characteristic strictly smaller than $\mOmega'$. Moreover $\mOmega''$ has at most 8 more edges than $\mOmega'$. Therefore, we have 
			\[ |E\mOmega''| + 8\chi(\mOmega'') \leq |E\mOmega'| + 8\chi(\mOmega') \] so that condition \eqref{cond:Omega_complexity} is satisfied by $\mOmega''$. Lastly, conditions \eqref{cond:Omega_vertex} and \eqref{cond:Omega_edge} are satisfied by $\hOmega''$ since $\hOmega'' \supset \hOmega'$.
			In conclusion, $\mOmega'$ was not of minimal Euler characteristic, a contradiction. See Figure \ref{fig:Omega_path} for an illustration.
			\begin{figure}
				\begin{center}\input{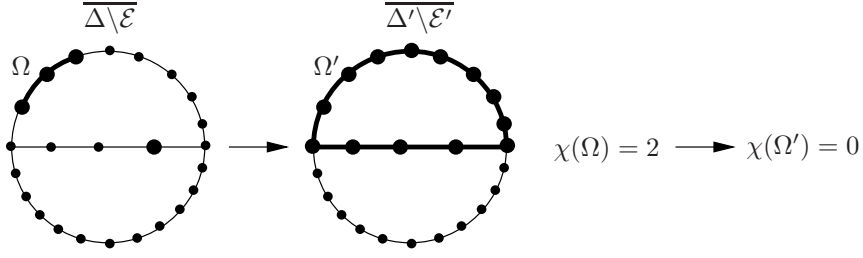}\end{center}
				\caption{Lemma \ref{lemma:Omega_path}: reducing $\chi(\mOmega)$.}
				\label{fig:Omega_path}
			\end{figure}	
		\end{proof}
		
		\begin{proposition}\label{prop:main_tech}
			Let $(\Dc,\pairOmega)$ be a tame marked decomposition of minimal complexity. Then $\Theta$ is folded.
		\end{proposition}
		
		\begin{proof} Assume by contradiction that $\Theta$ is not folded. Thus there are $e,e'\in E\Theta$ such that $\alpha(e)=\alpha(e')$ and $\ell(e) = \ell(e')$. Let $\Theta' = \Theta / e \sim e'$. We shall exhibit a tame marked decomposition $(\Dc',\pairOmega')$ of $\Theta'$ with the same primary complexity as $(\Dc,\pairOmega)$. This will immediately contradict the minimality of the complexity of $(\Dc,\pairOmega)$ as $c_6' = |E\Theta'| = |E\Theta| - 1 = c_6 - 1$.
			
			\setcounter{claim}{0}
			\begin{claim} \label{claim:folding_DD} If $e,e' \in \bar \Delta$ and \changed{$\tilde e, \tilde e' \in \Delta$} are any lifts of $e,e'$ respectively then \changed{$\tilde e$} and \changed{$\tilde e'$} lie in $F$. In particular there are lifts \changed{$\hat e, \hat e' \in \Gamma$} of $e, e'$ respectively.
			\end{claim}
			
			Suppose that $e, e' \in \bar \Delta$. Then neither $e$ nor $e'$ is a loop edge by condition \eqref{cond:no_loop_on_special_path}. Let \changed{$\tilde e$ and $\tilde e'$} be lifts in $\Delta$ of $e, e'$ respectively. Since $\Delta$ is folded \changed{$\alpha(\tilde e) \neq \alpha(\tilde e')$}. Therefore \changed{$\tilde e, \tilde e' \in \hOmega_3$} by Corollary~\ref{cor:dist}~\eqref{cor:dist_at_most_3}. Lastly Lemma~\ref{lemma:EF_EOmega} implies that \changed{$\tilde e, \tilde e' \in F$} which finishes the proof of Claim~\ref{claim:folding_DD}.

\smallskip \changed{The claim implies that either $e,e'\in\bar \Gamma$ or that, possibly after exchanging the roles of $e$ and $e'$, that $e\in\bar\Gamma$ and $e'\in \bar\Delta\backslash \bar\Gamma$. Note that in the second case $e'$ has a unique lift $\tilde e'$ to $\Delta$. Depending on which case we are in we define a decomposition $\Dc'$ of $\Theta'$ as follows. For any edge $f$ in $\bar\Gamma$ we denote its unique lift to $\Gamma$ by $\hat f$.}			

			\begin{figure}[htb]
				\center \input{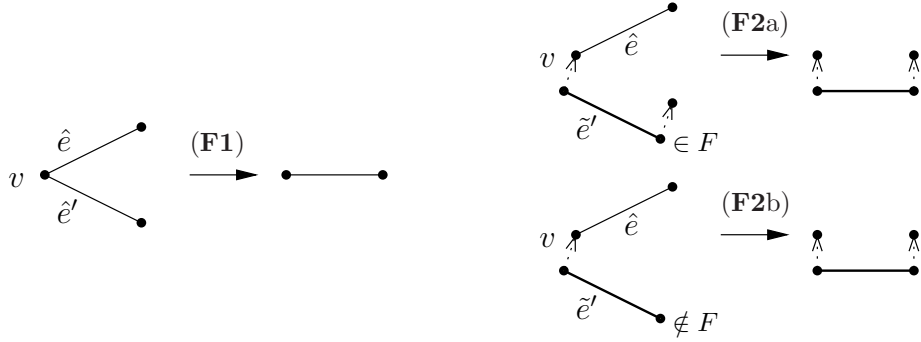}
				\caption{Proposition \ref{prop:main_tech}: Finding a decomposition of $\Theta'$.}
				\label{fig:types_of_folds}
			\end{figure}
			\begin{customenum}{\textbf{F}}
				\item \label{fold:GG} \changed{If $e,e'\in\bar\Gamma$ then}  $\Dc'$ is defined as follows: $\Delta' := \Delta$, $F' := F$, \changed{$\Gamma' := \Gamma / \hat e \sim \hat e'$} and $p':=\phi \circ p:F' \to \Gamma'$ where $\phi : \Gamma \to \Gamma'$ is the folding map.
				\item \label{fold:GD} If \changed{$e\in\bar\Gamma$ and $e'\in \bar\Delta\backslash \bar\Gamma$} then $\Dc'$ is defined as follows: in all cases $\Delta' := \Delta$.
				\begin{enumerate}
					\item If \changed{$\omega(\tilde e')\in F$ then $\Gamma' := (\Gamma \backslash \hat e) / \omega(\hat e) \sim p(\omega(\tilde e'))$, $F':= F$ and $p':= \phi \circ p : F' \to \Gamma'$ where $\phi: \Gamma \backslash \hat e \to \Gamma'$ is the obvious map. \label{fold:GDa}}
					\item If $\omega(\tilde e')\notin F$ then \changed{$\Gamma' := \Gamma \backslash \hat e$}, $F':= F \cup \omega(\tilde e')$ and $p':F' \to \Gamma'$ is given by $p'|_{F}:= p$ and \changed{$p'(\omega(\tilde e')):= \omega(\hat e)$}. \label{fold:GDb}
				\end{enumerate}
			\end{customenum}
			
			Note that $\betti(\Theta')\le\betti(\Theta)$, $\Delta'=\Delta$, $(\Theta')^f = \Theta^f$ and $(\Delta')^f=\Delta^f$. It follows that the primary complexity of $\Dc'$ is not greater than that of $\Dc$. Moreover $|E\Theta'| < |E\Theta|$ so that $c_6' < c_6$ and $\Dc'$ has strictly smaller complexity than $\Dc$. Also the potential $c_*'$ of $\Dc'$ is not greater than that of $\Dc$ so condition \eqref{cond:M_large} holds for $\Dc'$. Since $\Theta$ maps onto $\Theta'$ and since condition \eqref{cond:pi1surj} holds for $\Dc$, it also holds for $\Dc'$. Moreover, since \eqref{cond:two_vert_same_type} holds for $\Dc$ by Lemma~\ref{lemma:delta_12}, the fold did not identify two edges in the image of a single special path. Thus condition \eqref{cond:inj_edges_sppaths} holds for $\Dc'$ as \eqref{cond:inj_edges_sppaths} holds for $\Dc$.
			
			To conclude Proposition \ref{prop:main_tech} we need to define a marking $\hOmega'$ of $\Dc'$ that satisfies  \eqref{cond:Omega}. By Lemma \ref{lemma:Omega_path} it suffices to find a marking of $\Dc'$ satisfying conditions \eqref{cond:Omega_vertex}, \eqref{cond:Omega_edge} and \eqref{cond:Omega_complexity}. Let $\mOmega^1 \subset \overline{\Delta' \backslash \Ec'}$ be the image of $\mOmega \subset \overline{\Delta \backslash \Ec}$ under the folding map $\Theta \to \Theta'$. If at most one of the two edges $e$, $e'$ lies in $\bar \Delta$ and both $\omega(e), \omega(e')$ lie in $\bar \Delta$, then let $\mOmega' := \mOmega^1 \cup \{v'\}$ where $v'$ is the image of $\omega(e)$ and $\omega(e')$ under the fold (see Figure \ref{fig:step2_add_vertex}). Otherwise define $\mOmega' := \mOmega^1$.
			\begin{figure}
				\center\input{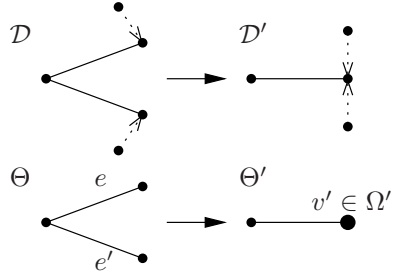}
				\caption{Proposition \ref{prop:main_tech}: constructing a new marking $\mOmega$ satisfying \eqref{cond:Omega}.}
				\label{fig:step2_add_vertex}
			\end{figure}
			
			\begin{claim} \label{claim:Omega_vertex} $ \hOmega'$ satisfies \eqref{cond:Omega_vertex}. 
			\end{claim}
			Let $x$ and $y$ be two vertices of \changed{$F'\subset\Delta'=\Delta$} such that $p'(x) = p'(y)$. \changed{We need to show that $x,y\in\hOmega_3'$. If}  $p(x) = p(y)$ then $x,y\in\hOmega_3$ since \eqref{cond:Omega_vertex} holds for $\mOmega$, the claim follows \changed{as  $\hOmega_3 \subset \hOmega_3'$}. \changed{Thus we can assume that $p(x)\neq p(y)$, i.e. that $p$ maps $\{x,y\}$ to $\{\omega(e),\omega(e')\}$}. Without loss of generality we assume $p(x)=\omega(e)$ and $p(y)=\omega(e')$. 
			
			If at most one of $e,e'$ lies in $\bar \Delta$ then $p(x) = p(y) = v' \in \mOmega'$ by construction of $\mOmega'$, so in particular $x,y \in \hOmega_3'$. If $e,e'\in \bar \Delta$, then \changed{let} $\tilde e$ and $\tilde e'$ be the (unique) lifts of $e$ and $e'$ in $\bar \Delta$ such that $\omega(\tilde e) = x$ and $\omega(\tilde e') = y$. By Claim \ref{claim:folding_DD} both $\tilde e$ and $\tilde e'$ lie in $\hOmega_3$, and in particular $\omega(\tilde e)$ and $\omega(\tilde e')=y$ lie in $\hOmega_3$. Recall that $\hOmega_3 \subset \hOmega_3'$ so that $x,y \in \hOmega_3'$. This finishes the proof that $ \hOmega'$ satisfies \eqref{cond:Omega_vertex}.
			
			\begin{claim} $\hOmega'$ satisfies \eqref{cond:Omega_edge}.
			\end{claim}
			
			By construction $EF = EF'$, and $\hOmega_3 \subset \hOmega_3'$ so that $EF'  = EF \subset E\hOmega_3 \subset E\hOmega_3'$.
			
			\begin{claim} $\mOmega'$ satisfies \eqref{cond:Omega_complexity}.
			\end{claim}
			
			\changed{We need to show that} that \[ |E\mOmega'| \leq 8(\chi(\Delta') - \chi(\bar\Delta') - \chi(\mOmega')).\]
			\changed{Recall} that by construction $\Delta' = \Delta$ and $|E\mOmega'| \leq |E\mOmega|$. \changed{As} $\mOmega$ satisfies \eqref{cond:Omega_complexity} this implies that $$|E\mOmega'| \leq |E\Omega|\le 8(\chi(\Delta) - \chi(\bar\Delta) - \chi(\mOmega))=8(\chi(\Delta') - \chi(\bar\Delta) - \chi(\mOmega)).$$ Thus it is enough to show that \[\chi(\bar\Delta') + \chi(\mOmega') \leq \chi(\bar\Delta) + \chi(\mOmega).\]
			
			Clearly we may focus on what happens to $e \cup e' $. We first show that $\chi(\bar \Delta') \leq \chi(\bar \Delta)$. Indeed, the only possibility for this not to hold is if $e \cup e' \subset \bar \Delta$ and $\omega(e) = \omega(e')$. But then either both $e$ and $e'$ are loop edges with the same base vertex and the same labeling, \changed{a contradiction to the validity of condition \eqref{cond:no_loop_on_special_path} for $(\Dc,\Omega)$}, or both $e$ and $e'$ are non-loop edges of $\bar \Delta$ with $\omega(e) = \omega(e')$ and the same labeling, which contradicts Lemma \ref{lemma:IIIA}.
			
			Since \changed{$\mOmega \subset \bar \Delta$} and $\mOmega^1$ is the image of $\mOmega$ in $\Theta'$ the arguments above also apply to $\mOmega^1$ and show that $\chi(\mOmega^1) \leq \chi(\mOmega)$. So if $\mOmega = \mOmega^1$ we are done. The only obstacle left is when $\chi(\mOmega') = \chi(\mOmega^1) + 1$, that is when one vertex is added to $\mOmega'$. Note that this happens whenever at most one of $e$, $e'$ is in $\bar \Delta$ and both $\omega(e), \omega(e')$ lie in $\bar \Delta$. But this in turns implies that $\chi(\bar \Delta') = \chi(\bar \Delta)-1$ so adding these contributions up yields 
			\[\chi(\bar\Delta') + \chi(\mOmega') = \chi(\bar \Delta)- 1 + \chi(\mOmega^1) + 1 \leq \chi(\bar \Delta) + \chi(\mOmega).\]
			
			This finishes to prove that $\mOmega$ satisfies conditions \eqref{cond:Omega_vertex}, \eqref{cond:Omega_edge} and \eqref{cond:Omega_complexity}. Thus up to replacing $\mOmega'$ by $\mOmega''$ as in Lemma \ref{lemma:Omega_path} we can assume $\mOmega'$ satisfies \eqref{cond:Omega}. In conclusion we have shown the existence of a tame marked decomposition $(\Dc',\pairOmega')$ of smaller complexity than $(\Dc,\pairOmega)$ a contradiction. This finishes the proof of Proposition \ref{prop:main_tech}.
		\end{proof}
		
	\section{The folded graph} \label{sec:folded_graph}
		
		\begin{lemma} \label{lemma:no_loop_in_Gamma} Let $(\Dc,\pairOmega)$ be a tame marked decomposition of minimal complexity. Then there is no loop edge in $\Gamma$.
		\end{lemma}
		\begin{proof}
			Suppose there is a loop edge $e$ in $\Gamma$. We claim that $\bar e \notin \bar \Delta$. Indeed, suppose that $\bar e$ has a lift $\tilde e \in F$. By definition of $F$, the edge $\tilde e$ is not a loop edge. Thus $\alpha(\tilde e) \neq \omega(\tilde e)$ are distinct vertices of $\Delta$ contained in a special path and such that $p(\alpha(\tilde e)) = p(\omega(\tilde e))$, a contradiction with condition \eqref{cond:two_vert_same_type}. Thus $\bar e \notin \bar \Delta$ and in particular $\bar e \notin \bar F$.
			
			We construct a tame marked decomposition $(\Dc',\pairOmega')$ of smaller complexity as follows. Let \changed{$\Delta':=\Delta \cup \Delta^0$ where $\Delta^0$ is a graph consisting of a single loop edge  $\tilde f$ with $\ell(\tilde f) = \ell(e)$ and a single vertex $\alpha(\tilde f)$.} Put $\Gamma' := \Gamma \backslash e$, $F' := F \cup \alpha(\tilde f)$ and $p':F' \to \Gamma'$ defined by $p'|_F := p$ and $p'(\alpha(\tilde f)) = \alpha(e)$. Finally, put $\mOmega':= \mOmega \subset \Delta'$ and apply unfolding of type \ref{unfold:artificial_segment} to $\alpha(\tilde e)$, so as to ensure that $\mOmega'$ satisfies condition \eqref{cond:Omega_vertex}. Clearly $(\Dc',\pairOmega')$ satisfies conditions \eqref{cond:Omega_edge}, \eqref{cond:Omega_complexity}, \eqref{cond:Omega_path}, \eqref{cond:pi1surj} and \eqref{cond:inj_edges_sppaths} since $(\Dc,\pairOmega)$ does. Remark that $\Theta' = \Theta$, $\chi(\Delta') = \chi(\Delta)$, $\cc(\Delta') = \cc(\Delta) +1$ and $|\Ec'| = |\Ec| +1$ so that $c_1' = c_1, c_2' = c_2$ and $c_*' = c_*$. Both $\Theta$ and $\Theta'$ are folded by Proposition \ref{prop:main_tech} so that $c_3' = |E(\Theta'^f \backslash \Delta'^f)| = |E(\Theta^f \backslash \Delta^f)| - 1 = c_3 - 1$. Thus $(\Dc',\pairOmega')$ is a tame marked decomposition of smaller complexity, which contradicts the minimality of the complexity of $(\Dc,\pairOmega)$.
		\end{proof}
		
		In order to state the next lemma, we extend some notation from special paths to any $S$-labeled path as follows. We say that a path $\gamma$ is of \emph{alternating type} if $\ell(\gamma)$ is an alternating word, i.e.\ a word of the form $stst \ldots stst$ or $stst \ldots tsts$. Moreover, the \emph{type} of $\gamma$ is defined as $\{s,t\}$.
		\begin{lemma}\label{lemma:alt_path_in_Omega}
			Let $(\Dc,\pairOmega)$ be a tame marked decomposition of minimal complexity. Let $\gamma$ be a path of alternating type $\{s,t\}$ in $\bar \Delta$ that does not meet any special path of the same type. Suppose all inner edges of $\gamma$ are non-loop edges.
			\begin{enumerate}
				\item Then all but possibly the first five and last five edges of $\gamma$ lie in $\mOmega$. \label{lemma:alt_path_in_O3_in_O}
				\item If moreover $l(\gamma) \geq 3$ and the extremal edges of $\gamma$ are loop edges, then the inner subpath of $\gamma$ lifts to a path in $\Gamma$. \label{lemma:alt_path_in_O3_LE}
			\end{enumerate}
		\end{lemma}

\begin{proof} Note first that any subpath of $\gamma$ whose edges and inner vertices lie outside $\mOmega_3$ is of length at most $2$. Indeed this is true as such a path lifts to a reduced path in $\Delta$ and any reduced path $\eta \subset \Delta$ of length greater than two \changedm{such that all inner edges of $\eta$ are non-loop edges} necessarily contains a subpath of length $2$ of some special path. 
It follows that all of $\gamma$ except possibly the initial and the terminal edge lies in $\mOmega_4$. \changed{As $\gamma$ is a reduced path it} follows from Lemma~\ref{homotopyOmega3}~(1) that all edges except possibly the initial and terminal subpath of length at most 5 lie in $\mOmega$. This proves~\eqref{lemma:alt_path_in_O3_in_O}.

			Suppose now that the extremal edges of $\gamma$ are loop edges. Let $\gamma'$ be the inner subpath of $\gamma$. Rephrasing  \eqref{lemma:alt_path_in_O3_LE} we seek to show that $\gamma' \subset \bar F$. Recall from Lemma~\ref{lemma:EF_EOmega} that $E\bar F = E\mOmega_3$ so it is sufficient to show that $\gamma' \subset \mOmega_3$. We have already proven in~\eqref{lemma:alt_path_in_O3_in_O} that $\gamma' \subset \mOmega_4$. In view of Lemma~\ref{homotopyOmega3}~(1) any path in $\mOmega_4$ connecting two vertices in $\mOmega_3$ is entirely contained in $\mOmega_3$. Thus it is enough to show that $\alpha(\gamma'), \omega(\gamma') \in \mOmega_3$. We concentrate on $\alpha(\gamma')$, the case of $\omega(\gamma')$ being identical. Let $e_1,e_2,e_3$ be the first three edges of $\gamma$. Thus $\alpha(\gamma') = \omega(e_1) = \alpha(e_2)$, the edge $e_1$ is a loop and $e_2$ is not.
			
			Assume that $\alpha(e_2) \notin \mOmega_3$. In particular $e_1, e_2$ are not in $\mOmega_3$ and therefore have unique lifts $\tilde e_1, \tilde e_2$, respectively, in $\Delta$. Note that $\omega(\tilde e_1) = \alpha(\tilde e_2)$ as otherwise condition \eqref{cond:Omega_vertex} would not be satisfied. Moreover, $\tilde e_1$ is a loop edge as otherwise $\alpha(\tilde e_1), \omega(\tilde e_1)$ would fail to satisfy Condition~\eqref{cond:no_loop_on_special_path} in Lemma~\ref{lemma:delta_12}. Thus the subpath $e_1,e_2$ lifts to a path $\tilde e_1,\tilde e_2$ in $\Delta$ where $\tilde e_1$ is a loop edge and $\tilde e_2$ is not. \changedm{Recall the observation at the beginning of the proof that no subpath of length $3$ of $\gamma$ lifts to a path in $\Delta$. Thus if we choose any lift $\tilde e_3\in \Delta$ of $e_3$ we have $\omega(\tilde e_2) \neq \alpha(\tilde e_3)$. Finally Corollary \ref{cor:dist}~\eqref{cor:dist_no_loop_edge} applies to $e := \tilde e_2\inv$, $x := \alpha(e)$ and $y := \alpha(\tilde e_3)$ since $\omega(e) = \alpha(\tilde e_2)$ is the basepoint of the loop edge $\tilde e_1$. In conclusion $\tilde e_2 \in \hOmega_3$ and in particular $\alpha(e_2) \in \mOmega_3$, a contradiction. Thus~\eqref{lemma:alt_path_in_O3_LE} holds.}
\end{proof}


		\begin{obs} \label{obs:vertex_lifts_in_Gamma}
			Let $(\Dc,\pairOmega)$ be a marked decomposition. We can always assume a vertex $v \in \Theta$ lifts to a vertex $\hat v$ in $\Gamma$. Indeed, if it is not the case, then $v$ lifts to a vertex $\tilde v$ in $\Delta \backslash F$. Then we can add a new vertex $\hat v$ to $\Gamma$, add $\tilde v$ to $F$ and declare $p(\tilde v) = \hat v$ without affecting tameness or complexity.
		\end{obs}
		
		\begin{lemma}\label{lemma:gamma_not_embedded} Let $(\Dc,\pairOmega)$ be a tame marked decomposition of minimal complexity.
			Suppose $\gamma$ is a path in $\Theta$ such that $\ell(\gamma) = s(ts)^{m_{st}-2}$. Then $\gamma$ is not embedded.
		\end{lemma}
		\begin{figure}
			\center \input{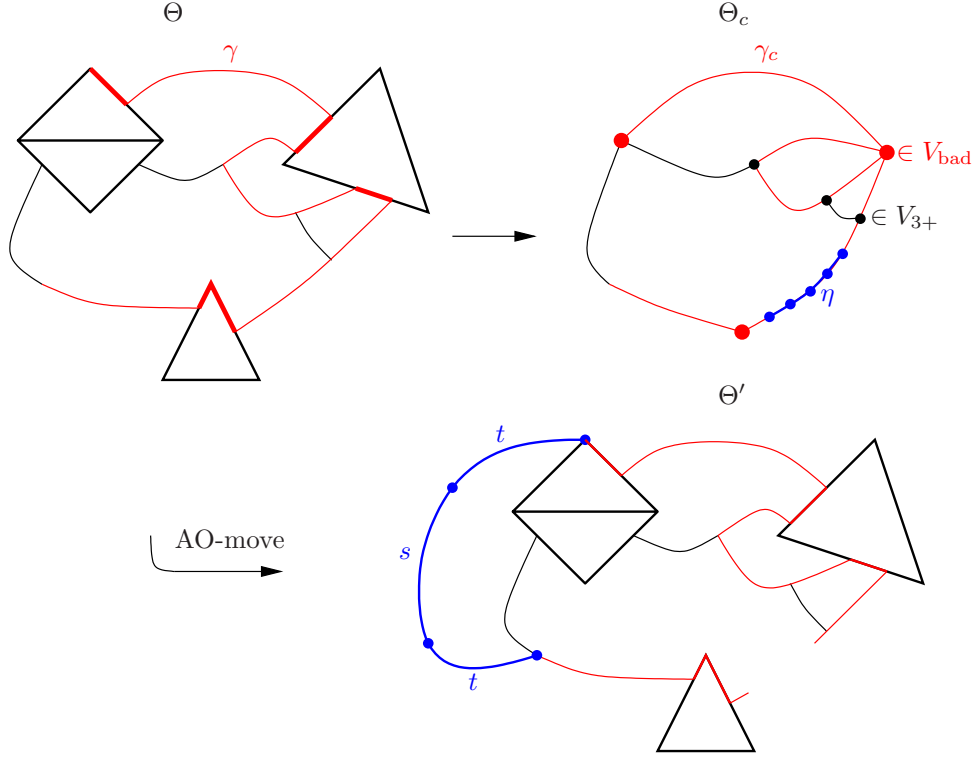}
			\caption{Lemma \ref{lemma:gamma_not_embedded}: doing an AO-move.}
			\label{fig:embedded_path}
		\end{figure}
		\begin{proof}
			\setcounter{claim}{0}
			Suppose that $\gamma$ is embedded in $\Theta$. Our goal is to show that since $m_{st}$ is \changed{sufficiently large by condition \eqref{cond:M_large}}, there is a subpath $\eta$ of $\gamma$ that has length $4$ lying outside of $\bar \Delta$ such that all inner vertices of $\eta$ have degree 2. If there is such a path, we can do the following AO-move: remove $\eta$ from $\Theta$ and glue a path of length three reading $tst$ connecting $\alpha(\gamma)$ and $\omega(\gamma)$. On the one hand, AO-moves do not affect condition \eqref{cond:pi1surj} by Lemma~\ref{lemma:fold_same_subgroup}. On the other hand this will reduce $c_4$ and thus produce a contradiction to the minimality of the complexity of $(\Dc,\pairOmega)$.
			
			We show that $\gamma$ contains 3 consecutive inner vertices of degree 2 in $\Theta \backslash \bar \Delta$. Let $\Theta_c$ be the graph obtained by collapsing each connected component of $\bar \Delta$ to a vertex. Let $V_{\bad}$ be the set of such such vertices in $\Theta_c$. Let $\gamma_c$ be the image of $\gamma$ under the map $\Theta \to \Theta_c$. Let $V_{3+}$ be the set of vertices that are not bad and of degree at least $3$ in $\Theta_c$. A \emph{visit} of a bad vertex $v \in V_{\bad}$ is a maximal subpath of $\gamma$ that is mapped to $v$. Let $C$ be the number of collapsed edges when mapping $\gamma \to \gamma_c$, and \changed{$N$} be the number of visits of bad vertices. The goal is to show that \[l(\gamma_c) \geq 3(\changed{N} + |V_{3+}|) + 4\] \changed{This is Claim~\ref{bound:length_gamma_c} below.} Indeed if this is the case, then there are three consecutive interior vertices of $\gamma_c$ that have degree $2$ and that are not bad vertices in $\Theta_c$. This immediately implies that $\gamma$ contains 3 consecutive inner vertices of degree 2 in $\Theta \backslash \bar \Delta$ as desired.
			
		\changed{	\begin{claim} $\gamma$ does not meet the image of a special path of type $\{s,t\}$. \label{claim:no_inter_sppath}
			\end{claim} 
			Suppose $\gamma$ does meet the image of a special path $\delta$ of type $\{s,t\}$. As $\Theta$ is folded $\gamma$ is contained in the subgraph consisting of $\delta$ and two loop edges $e$ and $f$ such that the path $\delta,f,\delta^{-1},e$ is labeled with the relation $(st)^{\pm m_{st}}$. As $\gamma$ is longer than $\delta$ it follows that $\gamma$ contains a loop edge, which is a contradiction.} 
		
			\begin{claim} $C \leq 10 \changed{N} + |E\mOmega|$. \label{claim:bound_collapsed_edges}
			\end{claim}
			Indeed, combining Claim \ref{claim:no_inter_sppath} and Lemma \ref{lemma:alt_path_in_Omega} any visit of a bad vertex contains at most 10 edges outside of $\mOmega$.
			
			\begin{claim} $\sum_{v \in V_{\bad}} d(v) \geq 2 \changed{N} - 2$. \label{claim:degree_bad}
			\end{claim}
			Consider the set of oriented edges of $\gamma \cup \gamma\inv$ whose initial vertex lies in a visit of a bad vertex but do not lie in a visit. There are at least $2\changed{N}-2$ such edges as $\gamma$ is simple. Each such edge is mapped to an edge starting at a bad vertex in $\Theta_c$, hence Claim \ref{claim:degree_bad} holds.
			
			\begin{claim} $\changed{N} \leq \betti(\Theta_c) + \cc(\Delta) - \frac{|V_{3+}|}{2}$. \label{claim:bound_visits}
			\end{claim}
			As $\Theta_c$ is connected we have the following
			\[ \betti(\Theta_c) = \sum_{v \in V\Theta_c} \frac{\deg(v) - 2}{2} + 1.\]
			We can assume $\Theta$ has no vertices of degree 1, so all vertices of degree 1 in $\Theta_c$ are bad. Thus the following holds
			\[ \betti(\Theta_c) = \sum_{v \in V_{\bad}} \frac{\deg(v)}{2} - |V_{\bad}| + \sum_{v \in V_{3+}} \frac{\deg(v) - 2}{2} + 1\]
			Applying Claim \ref{claim:degree_bad} and using the fact that $|V_{\bad}| = \cc(\bar \Delta) \leq \cc(\Delta)$ we get
			\[ \betti(\Theta_c) \geq \changed{N} - 1 - \cc(\Delta) + \frac{|V_{3+}|}{2} + 1\]
			which establishes Claim \ref{claim:bound_visits}.
			
			\begin{claim} $l(\gamma_c) \geq 3(\changed{N} + |V_{3+}|) + 4$ \label{bound:length_gamma_c}
			\end{claim}
			Recall from condition \eqref{cond:Omega_complexity} that \[|E\mOmega| \leq 8(\chi(\Delta) - \chi(\bar \Delta) - \chi(\mOmega)). \]
			As in the proof of Lemma~\ref{lemma:O3}  \changed{and using $b_0(\Delta)=b_0(\Delta\backslash \Ec)$ and $b_1(\bar\Delta\backslash\Ec)=b_1(\bar\Delta)-|\Ec|$} we observe that
			\[8(\chi(\Delta) - \chi(\bar \Delta) - \chi(\mOmega)) \leq 8(\cc(\Delta) + \betti(\bar \Delta) - |\Ec| + \betti(\mOmega) - 1). \]
			Combining this with Claims \ref{claim:bound_collapsed_edges} and \ref{claim:bound_visits} we have
			\begin{align*} C & \leq\changed{10N+|E\mOmega| = 13N - 3N} + |E\mOmega| \\
				&\leq 13\left(\betti(\Theta_c) + \cc(\Delta) - \frac{|V_{3+}|}{2}\right) - 3\changed{N} + 8(\cc(\Delta) + \betti(\bar \Delta) - |\Ec| + \betti(\mOmega) - 1) \\
				&\leq 21 \cc(\Delta) + \underbrace{13\betti(\Theta_c) + 8 \betti(\mOmega)}_{B_1} + \changed{\underbrace{(8\betti(\bar \Delta) -|\Ec|}_{B_2})} - \frac{13}{2}|V_{3+}| - 3\changed{N} - 8
			\end{align*}
			We claim that $B_1+\changed{B_2} \leq 21(\betti(\Theta) - |\Ec|)$.
			Indeed by construction of $\Theta_c$ we have $\betti(\Theta_c) = \betti(\Theta) - \betti(\bar \Delta)$. Moreover since $\mOmega$ is a subgraph of $\bar \Delta \backslash \Ec$ we have
			\[\betti(\mOmega) \leq \betti(\bar \Delta \backslash \Ec) = \betti(\bar \Delta) - |\Ec|.\]
			Thus \[B_1  = 13\betti(\Theta_c) + 8 \betti(\mOmega) \leq 13 (\betti(\Theta_c) + \betti(\mOmega)) \leq 13 (\betti(\Theta) - |\Ec|).\] 
			Moreover $\betti(\bar\Delta) \leq \betti(\Theta)$ so $\changed{B_2} \leq 8(\betti(\Theta) - |\Ec|)$.
			Combining the inequalities for $B_1$ and $\changed{B_2}$ we get $B_1+\changed{B_2} \leq 21(\betti(\Theta) - |\Ec|)$ as claimed.
			
			Using our previous estimate for $C$ and replacing $B_1 + \changed{B_2}$ we have
			\begin{align*} C &\leq 21 \cc(\Delta) + B_1 + \changed{B_2} - \frac{13}{2}|V_{3+}| - 3\changed{N} - 8\\
											&\leq 21 \underbrace{(\cc(\Delta) + \betti(\Theta) - |\Ec|)}_{=c_*} - \frac{13}{2}|V_{3+}| - 3\changed{N} - 8.
			\end{align*}
			
			Now recall that $\gamma$ is of length $2m_{st} - 3$ and that $l(\gamma_c) = l(\gamma) - C$ so we have
			\[ l(\gamma_c) \geq 2m_{st} - 3 - 21 c_* + \frac{13}{2}|V_{3+}| + 3\changed{N} + 8\] 
			By condition \eqref{cond:M_large} we have \[ 2m_{st} \geq 2.\Konst.2^{c_*} \geq 21 c_* - 1\] 
			and hence
			\[ l(\gamma_c) \geq 4 + \frac{13}{2}|V_{3+}| + 3\changed{N}\]
			which implies Claim \ref{bound:length_gamma_c}.
			
			
As we had reduced the claim of the Lemma to Claim~\ref{bound:length_gamma_c} this concludes the proof.		\end{proof}

		\begin{proposition}\label{proposition:Gammaf_is_R}
			Let $(\Dc,\pairOmega)$ be a tame marked decomposition of minimal complexity. Then \changedm{$\Theta = \Theta_S$}.
		\end{proposition}
		\begin{proof}\setcounter{claim}{0}
			\changedm{As $\Theta$ is folded by Proposition \ref{prop:main_tech} this} is equivalent to saying that for any $u \in S$ there is a loop edge based at $v_0$ in $\Theta$ labeled by $u$. We proceed by contradiction. Suppose there is an $u \in S$ such that there is no loop edge based at $v_0$ with label $u$. Choose a shortest loop $\gamma$ based at $v_0$ in $\Theta$ among those such that $\ell(\gamma) =_W u$. Such a loop exists by condition \eqref{cond:pi1surj}. It is clear that $\gamma$ must be reduced, otherwise it would not be of shortest length. Since $\Theta$ is folded, the word $w:=\ell(\gamma)$ is also reduced. By Lemma \ref{lemma:almost_whole_rel} $w$ contains a subword of the form $s(ts)^{m_{st}-2}$ and hence $\gamma$ has a subpath $\gamma'$ such that $\ell(\gamma') = s(ts)^{m_{st}-2}$.
			
			By Lemma \ref{lemma:gamma_not_embedded} the path $\gamma'$ is not embedded. \changed{ Let $v$ be a vertex in the image of $\gamma'$. As $\Theta$ is folded it follows that there is at most one edge adjacent to $v$ that has label $s$ and at most one edge adjacent to $v$ with label $t$. These edges may or may not be loop edges.} It follows that the image of $\gamma'$ in $\Theta$ is either an embedded path of length at least $m_{st}-2$ with a loop edge at one end, or an simple closed path of even length, or an embedded path of length at most $2m_{st}-5$ with loop edges at both ends.
			
			Claims~\ref{claim:add_translate}, \ref{claim:figure8} and \ref{claim:collapse_loop} assert that none of these situations can occur. Hence Proposition~\ref{proposition:Gammaf_is_R} immediately follows from these claims.
			
			\begin{claim} The path $\gamma'$ does not meet the image of a special path of type $\{s,t\}$.
			\end{claim}
			Indeed, suppose $\gamma'$ meets a special path of type $\{s,t\}$. Since $\Theta$ is folded the image of $\gamma'$ is contained in the union of the special path and \changed{two loop edges based at $\alpha(\delta)$ and $\omega(\delta)$}. By condition \eqref{cond:Delta_relator} $\gamma'$ has the same endpoints as a path $\gamma''$ reading the word $tst$. Moreover, replacing $\gamma'$ with $\gamma''$ in the path $\gamma$ does not change the \changed{element} it represents in $W$. As $m_{st}$ is at least $6$ the length of $\gamma'$ is at least $9$. This contradicts the fact that $\gamma$ is the shortest path in $\Theta$ reading a word that represents $u \in W$.
			
			\begin{claim} \label{claim:add_translate} The image of $\gamma'$ in $\Theta$ is not an embedded path $\gamma''$ of length at least $m_{st}-2$ with a loop edge $e$ based at the end of $\gamma''$.
			\end{claim}
			Assume by contradiction that the image of $\gamma'$ in $\Theta$ is as above. We show that the complexity of $(\Dc,\pairOmega)$ is not minimal.
			
			By Lemma \ref{lemma:no_loop_in_Gamma} the loop edge $e$ lies in $\bar \Delta$. Moreover it cannot be the image of a non-loop edge of $\Delta$ by condition \eqref{cond:two_vert_same_type}. So let $\tilde e$ be \changed{the unique} lift of $e$ in $\Delta$. Up to exchanging $s$ and $t$ we can assume $s = \ell(e)$. We construct a new decomposition $(\Dc',\pairOmega')$ as follows: (see Figure \ref{fig:adding_a_special_path}) let $\Delta' = (\Delta \cup \delta \cup f) / \sim$, where 
			\begin{itemize}
				\item $\delta$ is a path such that $\ell(\delta) = (ts)^{\frac{m_{st} - 1}{2}}$ if $m_{st}$ is odd or $\ell(\delta) = t(st)^{\frac{m_{st} - 2}{2}}$ is $m_{st}$ is even;
				\item $f$ is a loop edge with label $s$ if $m_{st}$ is even and $t$ if $m_{st}$ is odd;
				\item the identification $\sim$ is given by \changed{$\alpha(\delta) \sim \alpha(\tilde e)$ and $\omega(\delta)\sim \alpha(f)$.}
			\end{itemize}
			Put $\Gamma = \Gamma$, $F' = F$, $p' = p$ and $\mOmega'$ is the image of $\mOmega$ under the obvious map $\Theta \to \Theta'$. Declare $\delta$ to be a special path of $\Delta'$, so that $\Delta'$ is an $M$-special graph. Tameness of the marked decomposition $(\Dc',\pairOmega')$ follows immediately from that of $(\Dc,\pairOmega)$.
			
			$\Delta'$ has one more special path and one more loop edge than $\Delta$ so $c_1' = c_1$ and $c_2' = c_2$. Observe also that the image of $\gamma''$ in $\Theta'$ folds on $\bar \delta \cup \bar f$, so that $c_3' = |E(\Theta'^f \backslash \Delta'^f)| \leq |E(\Theta^f \backslash \Delta^f)| - |E(\gamma'' \backslash \bar \Delta)|$. Thus if $\gamma''$ is not entirely contained in $\bar \Delta$ then $c_3' < c_3$, a contradiction to the minimality of the complexity of $(\Dc,\pairOmega)$. Suppose from now on that $\gamma'' \subset \bar \Delta$. \changedm{If $\mOmega$ is empty then the map $\Delta \to \bar \Delta$ is an embedding and $\gamma''$ lifts to a path in $\Delta$. This is impossible as $\gamma''$ is of length at least $3$ and does not meet the image of a special path of the type $\{s,t\}$. Thus we can assume $\mOmega$ is nonempty.} Recall from the proof of Lemma~\ref{lemma:O3} that since $\mOmega$ is nonempty we have
			\[ |E\mOmega| \leq 16(c_* - 1) - 8.\]
			\changedm{In particular $c_* \geq 2$.} Also recall that since $\gamma''$ does not meet the image of a special path of the type $\{s,t\}$, all but the 10 edges of $\gamma''$ lie in $\mOmega$ by Lemma \ref{lemma:alt_path_in_Omega} \eqref{lemma:alt_path_in_O3_in_O}.
			On the other hand, by condition \eqref{cond:M_large} $m_{st} \geq \Konst.2^{c_*}$ so that
			\[ |E\mOmega| \geq |E(\mOmega \cap \gamma'')| \geq l(\gamma'') - 10 \geq m_{st} - 12 \geq \Konst.2^{c_*} - 12.\]
			These inequalities are incompatible for any $c_* > 1$. 
			\begin{figure}
				\center \input{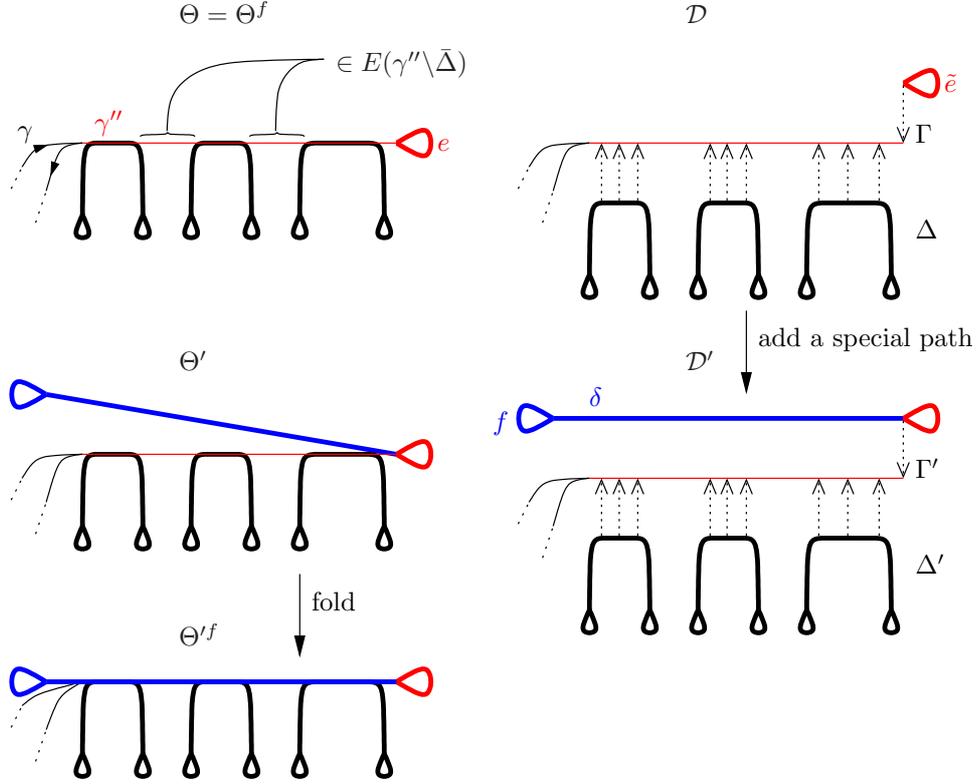}
				\caption{Adding a special path to $\Delta$.}
				\label{fig:adding_a_special_path}
			\end{figure}
			
			\begin{claim} \label{claim:figure8} The image of $\gamma'$ in $\Theta$ is not an embedded path $\gamma''$ of length at most $2m_{st}-5$ with loop edges $e$ and $e'$ at both ends.
			\end{claim}
			We argue by contradiction and assume that the image of $\gamma'$ in $\Theta$ is an embedded path $\gamma''$ of length at most $2m_{st}-5$ with loop edges $e$ and $e'$ at both ends.
			
			Suppose first that $l(\gamma'')=0$. Since $\Theta$ is folded $e, e'$ have different labels $\{s,t\}$. Therefore there is a path $\eta'$ of length 3 with the same endpoints as $\gamma'$ reading $tst$. As in Claim \ref{claim:collapse_loop} write $\gamma = \gamma_1 \cdot \gamma' \cdot \gamma_2$ and let $\eta = \gamma_1 \cdot \eta' \cdot \gamma_2$. Now $l(\eta) < l(\gamma)$ and $\ell(\eta) =_W \ell (\gamma) =_W u$. This is impossible as $\gamma$ is the shortest closed loop in $\Theta$ based at $v_0$ with labeling representing $u \in W$.
			
			From now on we assume $l(\gamma'')>0$. Let $x$ and $y$ be the basepoints of $e$ and $e'$ respectively. By Lemma \ref{lemma:no_loop_in_Gamma} the loop edges $e$ and $e'$ lift to loop edges $\tilde e$ and $\tilde e'$ in $\Delta$. Let also $\tilde x$ and $\tilde y$ be the basepoints of $\tilde e$ and $\tilde e'$ respectively. Our goal is to collapse $\gamma''$ and identify $\tilde x$ and $\tilde y$ in order to decrease complexity. In order to do that we need to lift $\gamma''$ to a path \changed{$\hat \gamma''$} in $\Gamma$. \changed{We will now show that we can assume that this is possible, possibly after replacing $(\Dc,\Omega)$ with another tame marked decomposition of the same primary complexity. We distinguish two cases.}
			
			\setcounter{case}{0}
			\begin{case} Suppose \changed{first that} $\gamma'' \subset \bar \Delta$.  Then Lemma \ref{lemma:alt_path_in_Omega} \eqref{lemma:alt_path_in_O3_LE} applies to the path $e\gamma''e'$ and we conclude that $\gamma''$ is entirely contained in $\bar F$. Therefore $\gamma''$ lifts to a path \changed{$\hat\gamma''$} in $\Gamma$.
			\end{case}
			
			\begin{case} \label{case:unfold_path} Suppose \changed{now} that there is an edge $f$ of $\gamma''$ that is not in $\bar \Delta$. We construct a new decomposition $(\Dc^1,\pairOmega^1)$ so that $\gamma''^1$ lies outside $\bar \Delta^1$ (except at its endpoints) as follows. Since $\gamma''$ is embedded, it can be rewritten as $\eta \cdot \changed{f} \cdot \eta'$ where $\eta$ and $\eta'$ are in $\Theta \backslash \{\changed{f}\}$. 
			
			By Observation \ref{obs:vertex_lifts_in_Gamma} we can assume that $\alpha(\gamma'')=x$ and \changed{$\omega(\gamma'')=y$} lift to vertices \changed{$\hat x$} and \changed{$\hat y$} in $\Gamma$.
			
			Put $\Gamma^1 := (\gamma''^1 \cup \Gamma \backslash \{e\}) / \sim$ where $\gamma''^1$ is a new path such that $\ell(\gamma''^1)=\ell(\gamma'')$ and the identification $\sim$ is given by $\alpha(\gamma''^1) \sim \changed{\hat x}$ and $\omega(\gamma''^1) \sim \changed{\hat y}$. Also put $\Delta^1 = \Delta$, $F^1 = F$, $p^1 = p$ and let $\mOmega^1$ be the obvious  image of $\mOmega$ in $\bar \Delta^1$. 
			Note that $\betti(\Theta)$ is the same as $\betti(\Theta')$ and that $\Theta'$ folds onto $\Theta$ so that the primary complexities remain the same, and that condition \eqref{cond:pi1surj} holds for $(\Dc^1,\pairOmega^1)$. Condition \eqref{cond:Omega}, \eqref{cond:inj_edges_sppaths} and \eqref{cond:M_large} clearly hold for $(\Dc^1,\pairOmega^1)$ as they hold for $(\Dc,\pairOmega)$. Therefore $(\Dc^1,\pairOmega^1)$ is a tame marked decomposition with the same primary complexity as $(\Dc,\pairOmega)$. \changed{Moreover} $\gamma''^1$ lifts to a path \changed{$\hat \gamma''^1$} in $\Gamma^1$, as desired. From now on we rename $(\Dc^1,\pairOmega^1)$ by $(\Dc,\pairOmega)$ \changed{and similarly drop all other superscipts equal to $1$}.
			\end{case}
			
			Clearly $\tilde x$ and $\tilde y$ are in $F$. By Lemma \ref{lemma:O3} \eqref{lemma:O3_small} $\tilde x$ and $\tilde y$ lie in different connected components of $F$. Let $(\Dc^u,\emptyset)$ be the tame marked decomposition obtained from $(\Dc,\pairOmega)$ by performing unfoldings of type \ref{unfold:cc_vertex} to each connected component of $F$, choosing $\tilde x$ and $\tilde y$ as distinguished vertices in their respective components, and for each vertex of the resulting forest except $\tilde x$ and $\tilde y$ perform unfolding of type \ref{unfold:artificial_segment}. Let $\tilde \gamma''^u $ denote the image of $\tilde \gamma''$ in $\Gamma^u$. Note that due to unfolding of type \ref{unfold:artificial_segment} no vertex $F^u$ is mapped by $p$ to $\tilde \gamma''^u$ except $\tilde x$ and $\tilde y$. Recall from Lemma \ref{lemma:unfold_same_cpx} that the primary complexity of $(\Dc^u,\emptyset)$ is equal to that of $(\Dc,\pairOmega)$.
			
			Now we construct a new decomposition $\Dc'$ as follows. Write $\tilde \gamma''^u = e_0,e_1, \ldots ,e_n$ and let $\Gamma'$ be obtained from $\Gamma^u$ by collapsing $\tilde \gamma''$ to a vertex $v$. Let $\Delta^1:=\Delta^u / \tilde x \sim \tilde y$. If $\ell(\tilde e) \neq \ell(\tilde e')$ let $\Delta':=\Delta_1$ and if $\ell(\tilde e) = \ell(\tilde e')$ let $\Delta'$ be obtained from $\Delta_1$ by relabeling $\ell'(\tilde e) = s$ and $\ell'(\tilde e') = t$. Let \changed{$\tilde x'$} denote the image of $\tilde x$ in $\Delta'$. Moreover, put $F' := F^u/ \tilde x \sim \tilde y$ and define $p' : F' \to \Gamma'$ by $p'(\changed{\tilde x'}) := v$ and $p'|_{F' \backslash \{\changed{\tilde x'}\}} := p^u \circ \phi$ where $\phi : \Gamma^u \to \Gamma'$ is the obvious quotient map (note that $\phi$ collapses edges, and so is not a graph map in the usual sense of this paper). We show that the marked decomposition $(\Dc',\emptyset)$ is tame. Indeed, conditions \eqref{cond:inj_edges_sppaths} and \eqref{cond:M_large} follow immediately from the tameness of $(\Dc^u, \emptyset)$. Moreover, $p'$ is injective since $p^u$ maps $F^u \backslash \{\tilde x, \tilde y\}$ injectively to $\Gamma^u \backslash \tilde \gamma''^u$. Therefore condition \eqref{cond:Omega} holds for $(\Dc',\emptyset)$. Lastly, letting $z$ be the image of \changed{$\tilde x'$} in $\Theta$ by construction of $\Dc'$ there are loop edges $e, e' \in \Theta'$ based at $z$ with labels $s,t$ respectively. Hence there is an obvious label-preserving map $\theta : \Theta^u \to \Theta'$, and in particular condition \eqref{cond:pi1surj} holds for $(\Dc',\emptyset)$.
			
			Note that $\betti(\Gamma') = \betti(\Gamma^u)$, that $\cc(F') = \cc(F^u) - 1$ and that $\chi(\Delta') = \chi(\Delta^u) -1$ so that $\betti(\Theta') = \betti(\Theta^u)$,  $c_1' = c_1$ and $c_2' = c_2$. \changed{Recall that $\Theta=\Theta^f$ by Proposition~\ref{prop:main_tech}.} Moreover $\theta:\Theta^u \to \Theta'$ induces a \changed{surjective} map $\Theta =\Theta^f= (\Theta^u)^f \to \Theta'^f$ \changedm{that maps $\bar \Delta = \Delta^f$ onto $\Delta'^f$ and that is not injective on edges. Therefore \[ c_3' = |E(\Theta'^f \backslash \Delta'^f)| \leq |E(\Theta^f \backslash \Delta^f)| = c_3\]
			and \[c_4' = |E\Theta'^f| < |E\Theta|= |E\Theta^f|= c_4.\]}  Hence $(\Dc',\emptyset)$ is a tame marked decomposition of smaller primary complexity than $(\Dc,\pairOmega)$, a contradiction.

			\begin{claim}\label{claim:collapse_loop} The image of $\gamma'$ in $\Theta$ is not a simple closed loop $\gamma''$ of even length.
			\end{claim}
			Suppose by contradiction that the image of $\gamma'$ in $\Theta$ is a simple closed loop $\gamma''$ of even length.
			
			First we show that $\gamma''$ is not of length $2$. Suppose it is the case. Since $\gamma'$ is of odd length, its endpoints are the two vertices of $\gamma''$. But since $\gamma''$ is of length 2 there is also a path $\eta'$ of length 3 with the same endpoints as $\gamma'$ reading $tst$. Write $\gamma = \gamma_1 \cdot \gamma' \cdot \gamma_2$ and let $\eta = \gamma_1 \cdot \eta' \cdot \gamma_2$. Now $l(\eta) < l(\gamma)$ and $\ell(\eta) =_W \ell (\gamma) =_W u$. This is impossible as $\gamma$ is the shortest closed loop in $\Theta$ based at $v_0$ with labeling representing $u \in W$.
			
			Thus $\gamma''$ is of length at least $4$. Our goal is to replace $\gamma''$ by a loop of length $2$, which will contradict the minimality of complexity. \changed{We first show that we can} lift $\gamma''$ to a closed loop \changed{$\hat \gamma''$ in $\Gamma$, possibly after replacing $(\Dc,\mOmega)$ with another tame marked decomposition of the same primary complexity}.			
			
			\setcounter{case}{0}
			\begin{case} Suppose first that $\gamma'' \subset \bar \Delta$. Then Lemma \ref{lemma:alt_path_in_Omega} applies and $\gamma''$ is entirely contained in $\mOmega$. Indeed, each edge of $\gamma''$ is not a loop edge and can be considered as an edge which is neither in the five first or the five last edges of the path $\gamma'' \cdot \gamma'' \cdot \gamma''\cdot\gamma''$. By Lemma \ref{lemma:EF_EOmega} all edges of $\mOmega$ are in $\bar F$ so that $\gamma''$ lifts to a closed path \changed{$\hat \gamma''$} in $\Gamma$.
			\end{case}
			
			\begin{case} Suppose now that there is an edge $f$ of $\gamma''$ that is not in $\bar \Delta$. Proceeding in exactly the same way as in case \ref{case:unfold_path} of claim \ref{claim:figure8} we construct a tame marked decomposition $(\Dc^1,\pairOmega^1)$ such that $\gamma''^1$ lies outside of $\bar \Delta^1$ (except possibly at its basepoint) without changing the primary complexity. \changed{In particular $\gamma''^1$ has a lift $\hat\gamma''^1$ in $\Gamma^u$.} From now on we rename $(\Dc^1,\pairOmega^1)$ by $(\Dc,\pairOmega)$ \changed{and similarly drop all other superscipts equal to $1$}.
			\end{case}
			
			Let $(\Dc^u,\emptyset)$ be the tame marked decomposition obtained from $(\Dc,\pairOmega)$ by performing unfoldings of type \ref{unfold:cc_vertex} to each connected component of $F$, and for each vertex of the resulting forest perform unfolding of type \ref{unfold:artificial_segment}. Let \changed{$\hat \gamma''^u$} denote the image of \changed{$\hat \gamma''$} in $\Gamma^u$. Note that due to unfolding of type \ref{unfold:artificial_segment} no vertex of \changed{$\hat \gamma''^u$} is in the image of $p^u$. Recall from Lemma \ref{lemma:unfold_same_cpx} that the primary complexity of $(\Dc^u,\emptyset)$ is equal to that of $(\Dc,\pairOmega)$.
			
			Now we construct a new decomposition $\Dc'$ as follows. Write $\changed{\hat \gamma''^u}= e_0, e_1, \ldots, e_n$ and let $\Gamma' = \Gamma^u / \sim$ where the identification is the equivalence relation generated by $e_i \sim e_{i+2}$ for $0 \leq i \leq n-2$. Put $\Delta' := \Delta^u$, $F' := F^u$ and $p' := p^u \circ \phi : F' \to \Gamma'$ where $\phi : \Gamma^u \to \Gamma'$ is the obvious map. Thus $\changed{\hat \gamma''^u}$ is mapped to a loop of length two in $\Gamma'$. We show that the marked decomposition $(\Dc',\emptyset)$ is tame. Indeed, conditions \eqref{cond:inj_edges_sppaths} and \eqref{cond:M_large} follow immediately from the tameness of $(\Dc^u, \emptyset)$. Moreover, $p'$ is injective since $p^u$ is and since no vertex of \changed{$\hat \gamma''^u$} is in the image of $p^u$. Therefore condition \eqref{cond:Omega} holds for $(\Dc',\emptyset)$. Lastly, the obvious maps $\Delta^u \to \Delta'$ and $\Gamma^u \to \Gamma'$ induce a label-preserving map $\theta:\Theta^u \to \Theta'$ so condition \eqref{cond:pi1surj} holds for $(\Dc',\emptyset)$.
			
			Note that $\betti(\Gamma') = \betti(\Gamma^u)$, that $\cc(F^u) = \cc(F')$ and that $\Delta' = \Delta^u$ so that $\betti(\Theta') = \betti(\Theta^u)$,  $c_1' = c_1$ and $c_2' = c_2$. Moreover $\theta: \Theta^u \to \Theta'$ induces a map $\Theta = (\Theta^u)^f \to \Theta'^f$ \changedm{that maps $\bar \Delta = \Delta^f$ onto $\Delta'^f$ and that is not injective on edges as $\theta$ maps $\gamma''$ to a path of length $2$. Therefore \[c_3' \leq |E(\Theta'^f \backslash \Delta'^f)| \leq |E(\Theta^f \backslash \Delta^f)| = c_3\] and \[c_4' = |E\Theta'^f| < |E\Theta^f| = c_4.\]} Hence $(\Dc',\emptyset)$ is a tame marked decomposition of smaller primary complexity than $(\Dc,\pairOmega)$, a contradiction.
		\end{proof}
		
		\changedm{We are finally ready to prove Theorem~\ref{thm:reformulation}.
		\begin{proof}[Proof of Theorem~\ref{thm:reformulation}] It clearly suffices to verify the claim for a tame marked decomposition $(\Dc,\pairOmega)$ of minimial complexity, which exists by Lemma~\ref{lemma:exist_minimalcpx_decomp}. By Proposition~\ref{proposition:Gammaf_is_R} we have $\Theta = \Theta_S$. We claim that $\Sppaths = \emptyset$. Assume by contradiction that there is a special path $\delta$ in the special graph $\Delta$. Then $\delta$ contains $2$ edges with the same label. But $\Theta_S$ has only one edge with a given label, contradicting condition~\eqref{cond:inj_edges_sppaths}. Thus $\Sppaths$ is empty, and $c_1 = \betti(\Theta_S) - |\Sppaths| = n - 0$.
		\end{proof}}

\appendix

	\section{Words in Coxeter groups}
		Throughout this section we fix a Coxeter matrix $M = (m_{st})_{s,t\in S}$ and a Coxeter group 
		\[W(M) = \langle S \mid (st)^{m_{st}}, s,t \in S\rangle.\]
		Our goal is to show that any reduced word representing $s\in S$ in $W$ must contain almost a whole relation as a subword.
		For $s,t \in S$ and $k \in \Nb$ we define the alternating word 
		\[\gamma_{st}(k) = 
			\begin{cases} 
				(st)^{\frac{k}{2}} &\text{ if $k$ is even} \\
				s(ts)^{\frac{k-1}{2}} &\text{ if $k$ is odd.} 
			\end{cases} \]
		
			
			We briefly recall Tits' solution of the word problem for Coxeter groups. Let $w$ be a word in the alphabet $S$. We say that $w'$ is obtained from $w$ by a \emph{cancellation} if $w = xssy$ and $w' = xy$ for two words $x,y$ and for some $s \in S$. We say that $w'$ is obtained from $w$ by a \emph{homotopy} if $w = x\gamma_{st}(m_{st})y$ and $w'=x\gamma_{ts}(m_{st})y$ for two words $x,y$ and for $s,t \in S$.
			
			Tits proved the following result \cite{Tits}:
			\begin{theorem}Let $w$ be a word representing $1$ in $W$. Then there is a sequence $w=w_0, w_1, ..., w_p = 1$ of words such that $w_{i+1}$ is obtained from $w_i$ by either a \emph{cancellation} or by a \emph{homotopy}. 
			\end{theorem}
			
			Remark that any cancellation or homotopy does not increase the length of the word. Therefore starting from a word $w$, there are finitely many words $\{x_1,\ldots,x_n\}$ that can be obtained by iterating these two operations. In particular it is possible to decide algorithmically whether $w$ represents $1$ in $W$ or not by checking if the trivial word is in $\{x_1,\ldots,x_n\}$. \changed{This algorithm is called Tits' algorithm.}
			
	\begin{lemma}\label{lemma:almost_whole_rel}
		Suppose $m_{st} \geq 4$ for any $s\neq t \in S$. Let $w$ be a reduced word of length at least $2$ such that $w =_W u$ for some $u \in S$. Then $w$ contains a subword of the form $\gamma_{st}(2m_{st}-3)$ for some $s\neq t \in S$.
	\end{lemma}
	\begin{proof}
		If $w = w'u$, then let $w'' := w'$. Otherwise let $w'' := wu$. In both cases $w''$ is a nontrivial reduced word and $w'' =_W 1$. It is clearly sufficient to show that $w''$ contains a subword of the form $\gamma_{st}(2m_{st}-3)$ that does not contain the last letter.
		
		Using Tits' algorithm for the word problem, we find \changed{a sequence} of reduced words $w''=w_0,w_1,w_2,...,w_q=1$ such that $w_{i+1}$ is obtained from $w_i$ by a homotopy, followed by as many cancellations as possible, until the word obtained is reduced. Note that going from $w_{i}$ to $w_{i+1}$ amounts to either removing a whole relation $\gamma_{st}(2m_{st})$ and canceling some more, or by replacing a maximal alternating subword $\gamma_{st}(2m_{st}-l)$ by the complementary word $\gamma_{ts}(l)$ for some $1\leq l \leq m_{st}$.
		
		For a reduced word $v$, we let $\kappa(v)$ be the minimal number of maximal alternating subwords needed to cover $v$. For brevity, we call $v_1,...,v_p$ a \emph{minimal cover} of $v$ if $v_1,...,v_p$ are maximal alternating subwords covering $v$ such that $p=\kappa(v)$. For such a minimal cover, note that $v_i$ and $v_{i+1}$ overlap by at most one letter, and there is at least one letter between $v_i$ and $v_{i+2}$. Consequently, if $v=xyz$ where $y$ is a maximal alternating subword of length at least $3$, then any minimal cover must contain $y$. In particular we have
		\begin{equation} 
			\kappa(v) = \kappa(x) +  1 + \kappa(z). \label{eq:kappa_three}
		\end{equation} 
		\changed{Moreover if} $v=xyz$ where $y=st$ is a maximal alternating subword of length $2$ and $v_1,...,v_p$ is any minimal cover of $v$ then 
		\begin{equation}
		 \kappa(v) = 
			\begin{cases} 
				\kappa(x) + \kappa(z) \text{ if } v = \cdot \overbrace{\cdot \cdot ut}^{v_{i}}\overbrace{su' \cdot \cdot}^{v_{i+1}} \cdot \text{ for } u,u' \in S \backslash \{s,t\} \\
				\kappa(x) +  1 + \kappa(z) \text{ in any other case.}
			\end{cases} \label{eq:kappa_two}
		\end{equation}
		We define \emph{the core of $v_i$}, denoted by $\hat v_i$ to be the subword of $v_i$ obtained by removing the possible overlapping letter with $v_{i-1}$ at the beginning and the possible overlapping letter with $v_{i+1}$ at the end, if $i = p$ we also remove the last letter of $v_i$ from the core.
		
		Since $w_0$ is non-trivial, $\kappa(w_0)>0$ and $w_q$ is empty, so $\kappa(w_q)=0$.	\changed{Observe that the} sequence $\kappa(w_i)$ is non-increasing. \changed{To see this we distinguish two cases. Suppose first that} $w_i = xyz$ and $w_{i+1} = xy'z$ where $y=\gamma_{st}(2m_{st}-l)$ is a maximal alternating subword and $y'=\gamma_{ts}(l)$ for some $1 \leq l \leq m_{st}$. Since \changed{$m_{st}\ge 4$} we know that $y$ is of length at least \changed{$4$}. Applying \eqref{eq:kappa_three} we have that $\kappa(w_i)=\kappa(x)+ 1 + \kappa(y) \geq \kappa(w_{i+1})$. \changed{In the the remaining case} $w_{i+1}$ is obtained from and $w_i$ by removing a relation and possibly \changed{followed by cancelations}\changed{. In this case it is trivial} that $\kappa(w_i) \geq \kappa(w_{i+1})$.
		
		Let $w_k$ be the first word in the sequence such that $\kappa(w_k) < \kappa(w_0)$. We can write $w_{k-1}=xyz$ \changed{ such that $w_k$ is obtained from $xy'z$ by cancellation (if $y'=1$)} where $y=\gamma_{st}(2m_{st}-l)$ and $y'=\gamma_{ts}(l)$ for some $0 \leq l \leq m_{st}$. We first show that $l \leq 2$. Suppose \changed{that $l\ge 3$. Thus both $y$ and $y'$ are} alternating words of length at least $3$. Applying \eqref{eq:kappa_three} yields that
		\[\kappa(w_{k-1}) = \kappa(x) +  1 + \kappa(z) = \kappa(w_{k}).\]
		This is in contradiction with the choice of $w_k$, so $l$ is at most $2$.
		
		Let $v_1,...,v_p$ be a cover of $w_{k-1}$ by maximal alternating words where $p = \kappa(w_{k-1})$. Let $m$ be the integer such that $v_m=y$. Consider the core $\hat v_m$ of $v_m$. We show that $\hat v_m$ is of length at least $2m_{st}-3$. If $l \leq 1$, that is if $v_m$ is of length at least $2m_{st}-1$ then clearly $\hat v_m$ is at least of length $2m_{st}-3$. Since $l < 3$ we are left with the case where $l = 2$. Recall that $w_{k-1} = xyz$ and $w_{k}=xy'z$, where $y'=ts$. Since $\kappa(w_k) < \kappa(w_{k-1})$ we know by \eqref{eq:kappa_two} that there are \changed{$u,u' \in S\backslash\{s,t\}$}  such that 
		\begin{align*} w_{k-1} &= \cdot \overbrace{\cdot \cdot u}^{v_{m-1}} \underbrace{st \cdots st}_{v_m} \overbrace{u' \cdot \cdot}^{v_{m+1}} \cdot  \\
								w_k &= \cdot \overbrace{\cdot \cdot u \makebox[0pt][l]{$\displaystyle{\underbrace{\phantom{\ ts}}_{y'}}$}\ t}^{v'_{m-1}}\ \overbrace{s\ u' \cdot \cdot}^{v'_{m+1}} \cdot
		\end{align*}
		Thus in this case $\hat v_m=v_m$ is of length $2m_{st}-2$.
		
		For every integer $0 \leq i \leq k-1$ let $u_1^i,..., u_p^i$ be a minimal cover of $w_i$, where $p = \kappa(w_0) = \kappa(w_{k-1})$. \changed{We show that for any $0 \leq i \leq k-1$ we have $\hat u_m^i = \gamma_{st}(2m_{st}-l)$ for some $l \leq 3$.} This will finish the proof as $\hat u_m^0$ is a subword of $w_0=w''$ which does not contain the last letter of $w''$. 
		
		\changed{Suppose to the contrary that for some $0 \leq i \leq k-1$ there exists no $l\leq 3$ such that $\hat u_m^i = \gamma_{st}(2m_{st}-l)$. Choose $i$ largest with this property. } Since $\kappa(w_{i}) = \kappa(w_{i+1})$  and $\kappa(w_i) = \kappa(x) + 1 +\kappa(z)$ we must have that $w_{i}=xyz$ and $w_{i+1}=xy'z$ \changedm{where $y$ is a maximal alternating subword of $w_i$ and $y'$ is a nontrivial alternating subword of $w_{i+1}$ in the same letters as $y$. Thus any minimal cover of $w_i$ is obtained by combining in an obvious way minimal covers of $x$ and $z$ with the maximal alternating subword $y$, and similarly any minimal cover of $w_{i+1}$ is obtained by combining minimal covers of $x$ and $z$ with a word covering $y'$. If $u_m^{i+1}$ does not cover $y'$ then $\hat u_m^{i+1}$ is a subword of $x$ or $z$. Therefore $\hat u_m^i$ contains $\hat u_m^{i+1}$ 
		(seen as subwords of either $x$ or $z$).} This is a contradiction since $\hat u_m^{i+1} = \gamma_{st}(2m_{st}-l)$  for $l \leq 3$ by definition of $i$. \changed{Thus we can assume that} $u_m^{i+1}=y'$. By definition of $i$ the word $\hat u_m^{i+1}=\gamma_{st}(2m_{st}-l)$ for $l\leq 3$. However $y'$ cannot be longer than $y$ so we have
		\[ 2m_{st}-3 \leq |y'| \leq |y| = 2m_{st}-|y'| \leq 3.\]
		Therefore $m_{st} \leq 3$ which is a contradiction. 
	\end{proof}
	\begin{remark} Alternatively, small cancellation theory can be used to prove Lemma \ref{lemma:almost_whole_rel}. We sketch such a proof. Realize $w$ by a loop in the Cayley graph, and let $D$ be a reduced van Kampen diagram for this loop where discs correspond to relations of the form $(st)^{m_{st}}$ where $s\neq t \in S$, and let $D'$ be a disk component of $D$. Then $D'$ is a subdivision of a disc into polygons such that each vertex has valence at least 3 and each face has at least 6 edges. Small cancellation theory then asserts that there must be at least two of these faces which have at most 3 edges in the interior of $D$, one of these yielding the desired subword in $w$.
	\end{remark}

	\normalsize

\end{document}